\documentclass[preprint,numbers,addressatend]{imsart}

\usepackage{amsthm,amsmath,natbib}
\usepackage{amsfonts,amssymb,graphicx}
\usepackage{color}

\startlocaldefs
\numberwithin{equation}{section}
\newtheorem{theo}{Theorem}[section]

\newtheorem{prop}[theo]{Proposition}
\newtheorem{lemma}[theo]{Lemma}
\newtheorem{defn}[theo]{Definition}
\newtheorem{assum}[theo]{Assumption}
\newtheorem{remark}[theo]{Remark}

\renewcommand{\labelenumi}{\alph{enumi}.)}

\newcommand{\cov}{\text{Cov} \mspace{1mu}}
\newcommand{\vertiii}[1]{{\left\vert\kern-0.25ex\left\vert\kern-0.25ex\left\vert #1 
 \right\vert\kern-0.25ex\right\vert\kern-0.25ex\right\vert}}

\usepackage{graphicx}
\usepackage{verbatim}

\endlocaldefs

\begin{document}

\begin{frontmatter}

\title{Likelihood Ratio Gradient Estimation \\ for Steady-State Parameters}

\runtitle{Likelihood Ratio Gradient Estimation}

\author{\fnms{Peter W.} \snm{Glynn}\ead[label=e2]{glynn@stanford.edu}}
\address{Department of Management Science \\ and Engineering \\  Stanford University \\ Stanford, CA 94305 \\ \printead{e2}}
\affiliation{Stanford University}
\and \hspace{4pt}
\author{\fnms{Mariana} \snm{Olvera-Cravioto}\corref{}\ead[label=e1]{molvera@berkeley.edu}}
\address{Department of Industrial Engineering  \\ and Operations Research \\ University of California Berkeley \\ Berkeley, CA 94720  \\ \printead{e1}}
\affiliation{University of California Berkeley}

\runauthor{P. Glynn and M. Olvera-Cravioto}

\begin{abstract}
We consider a discrete-time Markov chain $\boldsymbol{\Phi}$ on a general state-space {\sf X}, whose transition probabilities are parameterized by a real-valued vector $\boldsymbol{\theta}$. Under the assumption that $\boldsymbol{\Phi}$ is geometrically ergodic with corresponding stationary distribution $\pi(\boldsymbol{\theta})$, we are interested in estimating the gradient $\nabla \alpha(\boldsymbol{\theta})$ of the steady-state expectation 
$$\alpha(\boldsymbol{\theta}) = \pi( \boldsymbol{\theta}) f.$$
 To this end, we first give sufficient conditions for the differentiability of $\alpha(\boldsymbol{\theta})$ and for the calculation of its gradient via a sequence of finite horizon expectations. We then propose two different likelihood ratio estimators and analyze their limiting behavior. 
\end{abstract}



\end{frontmatter}

\section{Introduction}

Consider a discrete-time Markov chain $\boldsymbol{\Phi} = \{ \Phi_k: k \geq 0\}$ on a general state space {\sf X}  whose transition kernel $P(\boldsymbol{\theta}) = \{ P(\boldsymbol{\theta}, x, A) : x \in {\sf X}, \, A \subseteq {\sf X} \}$ is parameterized by a vector $\boldsymbol{\theta} \in \boldsymbol{\Theta} \subseteq \mathbb{R}^d$ of continuous parameters. We assume that $\boldsymbol{\Phi}$ has a unique invariant distribution $\pi(\boldsymbol{\theta}) = \{ \pi(\boldsymbol{\theta}, A): A \subseteq {\sf X}\}$ and that we are interested in computing the gradient of 
$$\alpha(\boldsymbol{\theta}) = \pi(\boldsymbol{\theta}) f = \int_{\sf X} f(x) \pi(\boldsymbol{\theta}, dx),$$
at a specific point $\boldsymbol{\theta}_0 \in \boldsymbol{\Theta}$, for some function $f$ such that $\pi(\boldsymbol{\theta}) |f| < \infty$. We consider throughout the paper the geometrically ergodic case, 
where the conditions for the existence of the gradient $\nabla \alpha(\boldsymbol{\theta})$ are stated more concisely and are easier to verify. We then focus on the analysis of  two different likelihood ratio estimators, exhibiting desirable limiting behavior, that can be used to approximate the gradient.  We refer the reader to \cite{Rhee_Glynn_16} for a more thorough analysis of the existence of the gradient under more general conditions.

We now give an informal description of the type of estimators that we study; the necessary assumptions will be made precise in the following section. From the strong law of large numbers we can expect that when $\boldsymbol{\Phi}$ evolves according to $P(\boldsymbol{\theta})$, then
$$\alpha_n = \frac{1}{n} \sum_{j=0}^{n-1} f(\Phi_j) \to \alpha(\boldsymbol{\theta}) \qquad P(\boldsymbol{\theta})-a.s.$$
as $n \to \infty$ for any initial distribution.  Moreover, if we assume that there exists a family of densities $\{ p(\boldsymbol{\theta},x,y): x, y \in {\sf X}\}$ such that  the transition probabilities satisfy
$$P(\boldsymbol{\theta}, x, dy) = p(\boldsymbol{\theta}, x, y) P(\boldsymbol{\theta}_0, x, dy),$$
then we can construct the likelihood ratio
$$L_n(\boldsymbol{\theta}) = \prod_{j=1}^n  p(\boldsymbol{\theta}, \Phi_{j-1}, \Phi_j), \qquad n \geq 1,$$
and use it to compute the expectation of $\alpha_n$ via the identity
\begin{equation} \label{eq:LR}
E_{\boldsymbol{\theta}} [\alpha_n] = E_{\boldsymbol{\theta}_0} \left[ \alpha_n L_n(\boldsymbol{\theta}) \right];
\end{equation}
here $E_{\boldsymbol{\theta}}[ \, \cdot \, ]$ denotes the expectation with respect to the transition probabilities $P(\boldsymbol{\theta})$ when the chain is started according to some fixed distribution $\mu$. Details regarding this identity can be found for example in \cite{Glynn_LEcuyer_95}, Theorem 1. Next, provided we have uniform integrability, we would have that 
$$E_{\boldsymbol{\theta}} [\alpha_n] \to \alpha(\boldsymbol{\theta})$$
as $n \to \infty$, and if we can further justify the exchange of derivative and expectation, then
\begin{equation} \label{eq:ExchangeDeriv}
\nabla E_{\boldsymbol{\theta}}[ \alpha_n ] = E_{\boldsymbol{\theta}_0} \left[ \alpha_n \nabla L_n(\boldsymbol{\theta}) \right] \to \nabla \alpha(\boldsymbol{\theta}) \qquad n \to \infty.
\end{equation}
We point out that $n E_{\boldsymbol{\theta}_0} [\alpha_n]$ also represents the finite horizon total cost incurred by $\boldsymbol{\Phi}$, and therefore, the calculation of its gradient is interesting in its own right, i.e., not only for its relation to $\nabla \alpha(\boldsymbol{\theta})$. 
For details on the estimation of gradients via likelihood ratios and other methods, as well as a variety of applications in finance, operations research and engineering, we refer the reader to \cite{LEcuyer_90, Glasserman_91, Fu_06}.

The observation made above suggest that one could think of using $\alpha_n \nabla L_n(\boldsymbol{\theta})$ as an estimator for $\nabla \alpha(\boldsymbol{\theta})$. Unfortunately, $\alpha_n \nabla L_n(\boldsymbol{\theta})$ fails to converge as $n \to \infty$; in fact, under some additional assumptions, $n^{-1/2} \alpha_n \nabla L_n(\boldsymbol{\theta})$ converges in distribution to a multivariate normal random variable (see Proposition~\ref{P.CLTbadestimator}). The first of our two proposed estimators, described in detail in Section~\ref{S.FirstEstimator}, uses $\nabla L_n(\theta_0)$ as a control variate to reduce the variance of $\alpha_n \nabla L_n(\boldsymbol{\theta}_0)$. The resulting estimator, after choosing the optimal control variate coefficient, is given by
\begin{equation} \label{eq:CVestimator}
(\alpha_n - \alpha(\boldsymbol{\theta}_0)) \nabla L_n(\boldsymbol{\theta}_0),
\end{equation}
and is shown to converge in Proposition~\ref{P.CLTgoodestimator}. An estimator of this type has been shown in \cite{Hash_Nun_Plec_Vlac_15} to be very successful in practice, where it was used to compute the sensitivities in reaction networks.

Our second estimator, described in detail in Section~\ref{S.SecondEstimator}, exploits the martingale structure of $\nabla L_n(\boldsymbol{\theta}_0)$ to obtain 
an alternative representation for $E_{\boldsymbol{\theta}_0} [ \alpha_n \nabla L_n(\boldsymbol{\theta}_0)]$ as the expectation of the discrete stochastic integral
\begin{equation} \label{eq:IntegralEstimator1}
\frac{1}{n} \sum_{k=1}^{n-1} \sum_{l=k}^{n-1} f(\boldsymbol{\Phi}_l) D_k,
\end{equation}
where the $\{D_k\}$ are the martingale differences in $\mathbb{R}^d$ satisfying $\nabla L_n(\boldsymbol{\theta}_0) = \sum_{k=1}^n D_k$. As is the case with $\alpha_n \nabla L_n(\boldsymbol{\theta})$, this estimator fails to converge on its own (see Proposition~\ref{P.CLTbadestimator2}), but can dramatically be improved by centering it with respect to $\alpha(\boldsymbol{\theta}_0)$. The optimized estimator takes the form 
\begin{equation} \label{eq:IntegralEstimator2}
\frac{1}{n} \sum_{k=1}^{n-1} \sum_{l=k}^{n-1} \left( f(\boldsymbol{\Phi}_l) -\alpha(\boldsymbol{\theta}_0) \right) D_k. 
\end{equation}
Moreover, the analysis of the asymptotic variance of \eqref{eq:CVestimator} and \eqref{eq:IntegralEstimator2}, included in Section~\ref{S.Remarks}, shows that \eqref{eq:IntegralEstimator2} is a better estimator than \eqref{eq:CVestimator}.

The first part of the paper establishes sufficient conditions on the Markov chain $\boldsymbol{\Phi}$ and the function $f$ under which $\nabla \alpha(\boldsymbol{\theta}_0)$ exists and the following limit holds
\begin{equation} \label{eq:ConvergenceToGradient}
E_{\boldsymbol{\theta}_0} \left[ \alpha_n \nabla L_n(\boldsymbol{\theta}_0) \right] \to \nabla \alpha(\boldsymbol{\theta}_0) \qquad n \to \infty.
\end{equation}
Conditions under which the exchange of derivative and expectation in \eqref{eq:ExchangeDeriv} is valid can be found in 
\cite{LEcuyer_95}, so we will not focus on this point. Once the convergence in \eqref{eq:ConvergenceToGradient} is established in Section \ref{S.Model}, we move on to the analysis of $\alpha_n \nabla L_n(\boldsymbol{\theta}_0)$ and the control variates estimator given in \eqref{eq:CVestimator}; the corresponding limit theorems are stated in Section~\ref{S.FirstEstimator}. The limit theorems for the the integral-type estimators given in \eqref{eq:IntegralEstimator1} and \eqref{eq:IntegralEstimator2} are included in Section~\ref{S.SecondEstimator}. To conclude the expository part of the paper, we compute in Section~\ref{S.Remarks} the asymptotic variance of our two proposed estimators. Finally, Section \ref{S.Proofs} contains the majority of the proofs.

\section{The model} \label{S.Model}

We consider throughout the paper a discrete-time Markov chain $\boldsymbol{\Phi} = \{ \Phi_k: k \geq 0\}$ on a general state space {\sf X} equipped with a countably generated $\sigma$-field $\mathcal{B}({\sf X})$, and governed by the transition kernel $P(\boldsymbol{\theta}) = \{ P(\boldsymbol{\theta}, x, A) : x \in {\sf X}, A \subseteq {\sf X} \}$. We assume that $\boldsymbol{\Theta} \subseteq \mathbb{R}^d$ is a family of continuous parameters for $P(\boldsymbol{\theta})$. Under the conditions given below, the Markov chain will possess a unique stationary distribution $\pi(\boldsymbol{\theta}) = \{ \pi(\boldsymbol{\theta}, A): A \subseteq {\sf X} \}$, and we are  interested in estimating the gradient of 
$$\alpha(\boldsymbol{\theta}) = \pi(\boldsymbol{\theta})f$$
at some fixed point $\boldsymbol{\theta}_0 \in \boldsymbol{\Theta}$, for some function $f: {\sf X} \to \mathbb{R}$ such that $\pi(\boldsymbol{\theta})|f| < \infty$.

In terms of notation, we use $\nu f$ to denote the expectation of $f$ with respect to measure $\nu$, that is,
$$\nu f = \int_{\sf X} f(x) \nu(dx).$$
Similarly, for any Markov transition kernel $P$ we use
$$P f(x) = \int_{\sf X} f(y) P(x, dy).$$
Whenever the context is clear we denote the gradient of a function $g$ at the point $\boldsymbol{\theta}_0$ by $\nabla g(\boldsymbol{\theta}_0)$, and when confusion may arise we will use the more precise notation $\left. \nabla g(\boldsymbol{\theta}) \right|_{\boldsymbol{\theta} = \boldsymbol{\theta}_0}$. The convention is to think of vectors as column vectors and to use ${\bf x}'$ to denote the transpose of ${\bf x}$. 

Before giving the main set of assumptions for the Markov chain $\boldsymbol{\Phi}$ we include for completeness some basic norm definitions. 

\begin{defn}
For $h: {\sf X} \to [1, \infty)$ let $L_h^\infty$ denote the space of all measurable functions $h$ on {\sf X} such that $|g(x)|/h(x)$ is bounded in $x$, equipped with the norm
$$|g|_h = \sup_{x \in {\sf X}} \frac{|g(x)|}{h(x)}.$$
\end{defn}

\begin{defn}
For $h: {\sf X} \to [1, \infty)$ define the $h$-{\em total variation norm} of any signed measure $\nu$ as
$$|| \nu ||_h = \sup_{g: |g| \leq h } \left| \nu g \right|.$$
\end{defn}

\begin{defn}
For a positive function $V: {\sf X} \to [1,\infty)$ we define the $V$-{\em operator norm} distance between two Markov transition kernels $P_1$ and $P_2$ as
$$\vertiii{P_1 - P_2}_V  = \sup_{h \in L_V^\infty, |h|_V = 1} \left| (P_1 - P_2) h \right|_V.$$
\end{defn}

{\bf Note:} It can be shown that the $h$-operator norm distance can be written in terms of the $h$-total variation norm as
$$\vertiii{P_1 - P_2}_V = \sup_{x \in {\sf X}} \frac{|| P_1(x, \cdot) - P_2(x, \cdot) ||_V}{V(x)}.$$ 

We can now state a set of sufficient conditions that will guarantee that $\nabla \alpha(\boldsymbol{\theta}_0)$ exists and that \eqref{eq:ConvergenceToGradient} holds.

\begin{assum} \label{A.Vuniformity}
Let $\boldsymbol{\Phi} = \{ \Phi_n: n \geq 0\}$ be a Markov chain taking values on ${\sf X}$ and having one-step transition probabilities $P(\boldsymbol{\theta}) = \{ P(\boldsymbol{\theta}, x, dy): x, y \in {\sf X} \}$, where $\boldsymbol{\theta} \in \boldsymbol{\Theta} \subseteq \mathbb{R}^d$. Fix $\epsilon > 0$ and define $B_\epsilon(\boldsymbol{\theta}_0) = \{ \boldsymbol{\theta} \in \boldsymbol{\Theta}: \max_{1\leq i\leq d} | \theta_i - \theta_{0,i} | < \epsilon \}$. 
\begin{enumerate} \renewcommand{\labelenumi}{\roman{enumi})}
\item Suppose that $\boldsymbol{\Phi}$ is $\psi$-irreducible for all $\boldsymbol{\theta} \in B_\epsilon(\boldsymbol{\theta}_0)$.

\item Suppose that for all $\boldsymbol{\theta} \in B_\epsilon(\boldsymbol{\theta}_0)$ there exist densities $p(\boldsymbol{\theta},x,y)$, differentiable at $\boldsymbol{\theta}_0$ and such that
$$P(\boldsymbol{\theta},x,dy) = p(\boldsymbol{\theta}, x, y) P(\boldsymbol{\theta}_0, x,dy).$$

\item Suppose there exists a set $K \subseteq \mathcal{B}({\sf X})$, $\delta > 0$, $m \in \mathbb{N}$ and a probability measure $\nu$ such that
$$P^m(\boldsymbol{\theta}, x, dy) \geq \delta \nu(dy) \qquad \text{for all } x \in K,$$
for all $\boldsymbol{\theta} \in B_\epsilon(\boldsymbol{\theta}_0)$.

\item For the set $K$ above suppose there exists a function $V: {\sf X} \to [1, \infty)$ and constants $0 < \lambda < 1$, $b < \infty$, such that
\begin{equation} \label{eq:drift}
P(\boldsymbol{\theta}) V (x) \leq \lambda V(x) + b 1_K(x)
\end{equation}
for all $\boldsymbol{\theta} \in B_\epsilon(\boldsymbol{\theta}_0)$.

\item Let $P^{(i)}(\boldsymbol{\theta}_0, x, dy ) = \left. \frac{\partial}{\partial \theta_i} p(\boldsymbol{\theta}, x, y) \right|_{\boldsymbol{\theta} = \boldsymbol{\theta}_0} P(\boldsymbol{\theta}_0, x, dy)$ and ${\bf e}_i$ be the vector that has a 1 in the $i$th component and zeros elsewhere. Assume that for each $1 \leq i \leq d$ we have $\vertiii{ P^{(i)}(\boldsymbol{\theta}_0)  }_V < \infty$ and 
$$\lim_{h \to 0} \vertiii{ \frac{P(\boldsymbol{\theta}_0 + h {\bf e}_i) - P(\boldsymbol{\theta}_0)}{h} - P^{(i)}(\boldsymbol{\theta}_0) }_V = 0.$$

\item Suppose  that $|g_{ii}|_V < \infty$ for each $1 \leq i \leq d$, where
$$g_{ii}(x) = \int_{{\sf X}} \left( \frac{\partial}{\partial \theta_i} p(\boldsymbol{\theta}_0, x,y)  \right)^2 P(\boldsymbol{\theta}_0, x, dy).$$


\item Suppose $|f|_{\sqrt{V}} < \infty$. 

\end{enumerate}
\end{assum}

\begin{remark}
\begin{itemize}
\item[i)] By iterating \eqref{eq:drift} we obtain
$$E_{\boldsymbol{\theta}, x} [ V(\Phi_k) ] = P^k(\boldsymbol{\theta}) V(x) \leq \lambda^k V(x) + b \sum_{i=0}^{k-1} \lambda^i < \infty$$
for all $x \in {\sf X}$ and all $k \in \mathbb{N}$. 
\item[ii)] A set $K$ satisfying Assumption~\ref{A.Vuniformity}(iii) is said to be a {\em small set}. 
\end{itemize}
\end{remark}

The conditions in Assumption \ref{A.Vuniformity}, which essentially impose geometric ergodicity (see e.g., \cite{Meyn_Tweedie}) of the chain $\boldsymbol{\Phi}$, are not necessary for the main convergence result of this section (Theorem \ref{T.Convergence}), but have the advantage of allowing us to keep the arguments concise and focus on the estimators in the following sections. A similar set of conditions has been used in \cite{Hei_Hor_Wei_06} (see Section~4.1). More general conditions ensuring the existence of the gradient outside of the geometric ergodicity setting can be found in \cite{Hei_Hor_09}, and more recently, in \cite{Rhee_Glynn_16}.

We will now proceed to give some properties of $\boldsymbol{\Phi}$, for which we will need the following definition. Proofs not included immediately after the corresponding statement can be found in Section~\ref{S.Proofs}.

\begin{defn}
We say that the Markov chain $\boldsymbol{\Phi}$ is $h$-{\em ergodic} if $h: {\sf X} \to [1, \infty)$ and
\begin{enumerate} \renewcommand{\labelenumi}{\roman{enumi})}
\item $\{ \Phi_k: k \geq 0 \}$ is positive Harris recurrent with invariant probability $\pi$.
\item the expectation $\pi h$ is finite
\item for every initial condition $x \in {\sf X}$,
$$\lim_{k \to \infty} || P^k(x, \cdot) - \pi ||_h = 0.$$
\end{enumerate}
\end{defn}

\begin{lemma} \label{L.ChainProperties}
Under Assumption \ref{A.Vuniformity}, the Markov chain $\boldsymbol{\Phi} = \{ \Phi_k: k \geq 0 \}$ is $V$-ergodic for each $\boldsymbol{\theta} \in B_\epsilon(\boldsymbol{\theta}_0)$. Furthermore, for all $1 \leq i \leq d$,
$$\lim_{h \to 0} \left|\left| \pi(\boldsymbol{\theta}_0 + h {\bf e}_i) - \pi(\boldsymbol{\theta}_0) \right|\right|_V = 0.$$
\end{lemma}

\begin{proof}
Fix $\boldsymbol{\theta} \in B_\epsilon(\boldsymbol{\theta}_0)$. By Assumption \ref{A.Vuniformity}(iv)
$$P(\boldsymbol{\theta}) \hat V(x) \leq \hat V(x) - V(x) + \hat b 1_K(x),$$
where $\hat V = (1-\lambda)^{-1} V$, $\hat b = (1-\lambda)^{-1} b$ and $K$ is a small set. Then, by Theorem 14.2.6 in \cite{Meyn_Tweedie} $\boldsymbol{\Phi}$ is $V$-regular, which in turn implies, by Theorem 14.3.3 in the same reference, that $\boldsymbol{\Phi}$ is $V$-ergodic. To establish the convergence in $V$-norm of the invariant probabilities first note that Assumption~\ref{A.Vuniformity}(v) yields
$$\lim_{h \to 0} \vertiii{ P(\boldsymbol{\theta}_0 + h {\bf e}_i) - P(\boldsymbol{\theta}_0) }_V = 0,$$
from where it follows that $\left|\left| \pi(\boldsymbol{\theta}_0 + h {\bf e}_i) - \pi(\boldsymbol{\theta}_0) \right|\right|_V \to 0$ as $h \to 0$ (see Section 4.2 in \cite{Glynn_Meyn_96}). 
\end{proof}

\bigskip

The main idea behind the analysis of the gradient of the likelihood ratio $L_n(\boldsymbol{\theta})$ is that under appropriate conditions each of its components is a square integrable martingale with respect to the family of filtrations generated by $\boldsymbol{\Phi}$. The next lemma makes this statement precise; its proof can be found in Section~\ref{S.Proofs}.

\bigskip

\begin{lemma} \label{L.MG_cond}
Suppose that Assumption \ref{A.Vuniformity} is satisfied. Define
\begin{align*}
D_j^k &= \left. \frac{\partial}{\partial \theta_k} \log p(\boldsymbol{\theta}, \Phi_{j-1}, \Phi_j) \right|_{\boldsymbol{\theta} = \boldsymbol{\theta}_0} = \frac{\partial}{\partial \theta_k} p(\boldsymbol{\theta}_0, \Phi_{j-1}, \Phi_j) , \qquad j = 1, 2, \dots, 
\end{align*}
and $D_j = (D_j^1, \dots, D_j^d)'$; let $\mathcal{F}_j$ denote the $\sigma$-field generated by $\Phi_0, \Phi_1, \dots, \Phi_j$. Then,  under $P(\boldsymbol{\theta}_0)$,
$$\nabla L_n(\boldsymbol{\theta}_0) = \sum_{j=1}^n D_j$$
is a square-integrable martingale in $\mathbb{R}^d$, that is, $M_n^k = \sum_{j=1}^n D_j^k$ $(M_0^k \equiv 0)$ is a square integrable-martingale adapted to $\mathcal{F}_k$ for each $k = 1, \dots, d$. 
\end{lemma}

\bigskip

The analysis of $\alpha_n \nabla L_n(\boldsymbol{\theta}_0)$ and of its expectation is based on a second martingale, one constructed via a solution $\hat f$ to Poisson's equation:
\begin{equation} \label{eq:PoissonEq}
\hat f - P(\boldsymbol{\theta}_0) \hat f = f - \pi(\boldsymbol{\theta}_0) f.
\end{equation}
Note that if this solution exists then the centered estimator $\alpha_n - \alpha(\boldsymbol{\theta}_0)$ can be written as follows:
\begin{align}
n(\alpha_n - \alpha(\boldsymbol{\theta}_0)) &= \sum_{k=1}^{n} \left( f(\Phi_{k-1}) - \pi(\boldsymbol{\theta}_0)f \right) \notag \\
&= \sum_{k=1}^{n} \left( \hat f(\Phi_{k-1}) - P(\boldsymbol{\theta}_0) \hat f(\Phi_{k-1}) \right) \notag \\
&= \hat f(\Phi_0) - \hat f(\Phi_n) + \sum_{k=1}^{n} \left( \hat f(\Phi_{k}) - P(\boldsymbol{\theta}_0) \hat f(\Phi_{k-1}) \right) , \label{eq:alphas}
\end{align}
where the terms $\hat f(\Phi_{k}) - P(\boldsymbol{\theta}_0) \hat f(\Phi_{k-1})$ can be shown to be martingale differences. It follows that provided $E_{\boldsymbol{\theta}_0} [ | \hat f(\Phi_0) - \hat f(\Phi_n) |]/n \to 0$ as $n \to \infty$, we have that $E_{\boldsymbol{\theta}_0} [ \alpha_n \nabla L_n(\boldsymbol{\theta}_0)] = E_{\boldsymbol{\theta}_0} [ (\alpha_n -  \alpha(\boldsymbol{\theta}_0)) \nabla L_n(\boldsymbol{\theta}_0)]$ is the expectation of a product of two martingales. The lemma below gives precise properties of this second martingale. 

\begin{lemma} \label{L.Poisson}
Suppose that Assumption \ref{A.Vuniformity} is satisfied, then $\pi(\boldsymbol{\theta}_0)V < \infty$, $\pi(\boldsymbol{\theta}_0)f^2 < \infty$ and a solution $\hat f$ to Poisson's equation \eqref{eq:PoissonEq} satisfying $|\hat f| \leq c_1 \sqrt{V}$ for some constant $c_1 < \infty$ exists. Moreover,   under $P(\boldsymbol{\theta}_0)$, $Z_n = \sum_{k=1}^n \left( \hat f(\Phi_k) - P(\boldsymbol{\theta}_0) \hat f(\Phi_{k-1}) \right)$ is a square-integrable martingale adapted to $\mathcal{F}_k = \sigma(\Phi_0, \Phi_1, \dots, \Phi_k)$. 
\end{lemma}

\bigskip

We are now ready to state our result for the convergence in \eqref{eq:ConvergenceToGradient}. 

\bigskip

\begin{theo} \label{T.Convergence}
Under Assumption \ref{A.Vuniformity}, $\alpha(\boldsymbol{\theta})$ is differentiable at $\boldsymbol{\theta}_0$ and
$$E_{\boldsymbol{\theta}_0} \left[\alpha_n \nabla L_n(\boldsymbol{\theta}_0) \right] \to \nabla \alpha(\boldsymbol{\theta}_0) \qquad n \to \infty.$$
\end{theo}

\section{A first likelihood-ratio estimator} \label{S.FirstEstimator}

In view of Theorem \ref{T.Convergence}, the remainder of the paper is devoted to the analysis of potential estimators for $E_{\boldsymbol{\theta}_0} \left[ \alpha_n \nabla L_n(\boldsymbol{\theta}_0)\right]$. An obvious first choice would be to consider 
\begin{equation} \label{eq:FirstBadEstimator}
\alpha_n \nabla L_n(\boldsymbol{\theta}_0)
\end{equation}
itself.  Unfortunately, as mentioned in the introduction, $\alpha_n \nabla L_n(\boldsymbol{\theta}_0)$ does not converge to an a.s. finite random variable; in fact, under additional assumptions, $n^{-1/2} \alpha_n \nabla L_n(\boldsymbol{\theta}_0)$ converges in distribution to a multivariate normal random vector, which implies that $\alpha_n \nabla L_n(\boldsymbol{\theta}_0) $ fails to converge at all. This observation is a simple consequence of the following weak convergence result, which will also be helpful in the analysis of the estimators considered in Section \ref{S.SecondEstimator}.  

Throughout the rest of the paper let $D([0,1], \mathbb{R}^d)$ denote the space of right-continuous $\mathbb{R}^d$-valued functions on $[0,1]^d$ with left limits equipped with the standard Skorohod topology; we use $\Rightarrow$ to denote weak  convergence.  From now on, the Markov chain $\boldsymbol{\Phi}$ is always assumed to evolve according to $P(\boldsymbol{\theta}_0)$. 

\begin{theo} \label{T.FCLT}
Suppose that Assumption \ref{A.Vuniformity} is satisfied and let $\hat f$ be the solution to Poisson's equation \eqref{eq:PoissonEq} from Lemma \ref{L.Poisson}. Define the functions $g_{ij}$ according to 
\begin{align*}
g_{ij}(x) &= \int_{\sf X} \frac{\partial}{\partial \theta_i}  p(\boldsymbol{\theta}_0, x, y) \, \frac{\partial}{\partial \theta_j}  p(\boldsymbol{\theta}_0, x, y) P(\boldsymbol{\theta}_0,x, dy), \qquad 1 \leq i, j \leq d, \\
g_{i0}(x) = g_{0i}(x) &= \int_{\sf X} \hat f(y) \frac{\partial}{\partial \theta_i} p(\boldsymbol{\theta}_0, x, y) P(\boldsymbol{\theta}_0,x, dy), \qquad 1 \leq i \leq d, \\
g_{00}(x) &= \int_{\sf X} \left( \hat f(y) - P(\boldsymbol{\theta}_0) \hat f(x) \right)^2 P(\boldsymbol{\theta}_0,x, dy) = P(\boldsymbol{\theta}_0) \hat f^2(x) - \left( P(\boldsymbol{\theta}_0) \hat f(x) \right)^2.
\end{align*}
Let $G(x) \in \mathbb{R}^{(d+1)\times(d+1)}$ be the matrix whose $(i,j)$th element is $g_{ij}(x)$ for $0 \leq i,j \leq d$. Then, 
$$ \left( n^{-1/2} \lfloor n \cdot \rfloor (\alpha_{\lfloor n \cdot \rfloor} - \alpha(\boldsymbol{\theta}_0)), \, n^{-1/2} \nabla L_{\lfloor n \cdot \rfloor}(\boldsymbol{\theta}_0)' \right) \Rightarrow B' \qquad n \to \infty,$$
in $D([0,1], \mathbb{R}^{d+1})$, where $B(t) = (B_0(t), B_1(t), \dots, B_d(t))'$ is a $(d+1)$-dimensional mean zero Brownian motion with covariance matrix $\pi(\boldsymbol{\theta}_0) G = (\pi(\boldsymbol{\theta}_0) g_{ij})$. 
\end{theo}

In view of this theorem we have the following result for $n^{-1/2} \alpha_n \nabla L_n(\boldsymbol{\theta}_0)$.

\begin{prop} \label{P.CLTbadestimator}
Suppose that Assumption \ref{A.Vuniformity} is satisfied and define for $1 \leq i,j \leq d$ the functions $g_{ij}$ according to Theorem \ref{T.FCLT}. Then,  
$$n^{-1/2} \alpha_n \nabla L_n(\boldsymbol{\theta}_0) \Rightarrow \alpha(\boldsymbol{\theta}_0) Z \qquad n \to \infty,$$
where $Z $ is a $d-$dimensional multivariate normal random vector having mean zero and covariance matrix $\Sigma = (\sigma_{ij})$, where $\sigma_{ij}  = \pi(\boldsymbol{\theta}_0) g_{ij}$ for $1\leq i,j \leq d$. 
\end{prop}

\begin{proof}
By Lemma \ref{L.ChainProperties} $\boldsymbol{\Phi}$ is $V$-ergodic, and since $|f|_{\sqrt{V}} < \infty$ we have $\pi(\boldsymbol{\theta}_0) |f| < \infty$. Then, by Theorem 17.0.1 in \cite{Meyn_Tweedie},
$$\lim_{n \to \infty} \alpha_n = \lim_{n \to \infty}  \frac{1}{n} \sum_{j=0}^{n-1} f(\Phi_j) = \pi(\boldsymbol{\theta}_0) f = \alpha(\boldsymbol{\theta}_0) \qquad \text{a.s. } P(\boldsymbol{\theta}_0).$$
By Theorem \ref{T.FCLT} we have
$$n^{-1/2} \nabla L_n(\boldsymbol{\theta}_0) \Rightarrow B(1),$$
where $B(t) = (B_1(t), \dots, B_d(t))'$ is a $d$-dimensional mean zero Brownian motion with covariance matrix $\Sigma$.  It follows by Slutsky's lemma that  
$$n^{-1/2} \alpha_n \nabla L_n(\boldsymbol{\theta}_0) \Rightarrow \alpha(\boldsymbol{\theta}_0) Z \qquad n \to \infty,$$
where $Z = B(1)$.
\end{proof}

\bigskip

Since $\alpha_n \nabla L_n(\boldsymbol{\theta}_0)$ does not converge as $n \to \infty$, we can define a new estimator with smaller variance by using as a control variate $\nabla L_n(\boldsymbol{\theta}_0)$, that is, we seek an estimator of the form
$$Y(C) \triangleq \alpha_n \nabla L_n(\boldsymbol{\theta}_0) + C \nabla L_n(\boldsymbol{\theta}_0),$$
where $C$ is a $d\times d$ constant matrix. Let $\Sigma_{Y(C)}$ be the covariance matrix of $Y(C)$, 
$$\Sigma_{Y(C)} = E_{\boldsymbol{\theta}_0} [(Y(C) - E_{\boldsymbol{\theta}_0}[Y(C)])(Y(C)- E_{\boldsymbol{\theta}_0}[Y(C)])'].$$
Our goal is to minimize the so-called generalized variance of $Y(C)$, defined as the determinant of $\Sigma_{Y(C)}$. The optimal choice for $C$ is given by
$$C_n^* = E_{\boldsymbol{\theta}_0}[(\alpha_n\nabla L_n(\boldsymbol{\theta}_0) - E_{\boldsymbol{\theta}_0}[\alpha_n \nabla L_n(\boldsymbol{\theta}_0)])(\nabla L_n(\boldsymbol{\theta}_0))'] \left( E_{\boldsymbol{\theta}_0}[(\nabla L_n(\boldsymbol{\theta}_0))(\nabla L_n(\boldsymbol{\theta}_0))'] \right)^{-1}$$
(see \cite{Rubin_Marc_85}). 
In the notation of Proposition \ref{P.CLTbadestimator}, 
$$C_n^* = E_{\boldsymbol{\theta}_0}[(\alpha_n M_n - E_{\boldsymbol{\theta}_0}[\alpha_n M_n]) M_n'] \left( E_{\boldsymbol{\theta}_0}[M_n M_n'] \right)^{-1} = E_{\boldsymbol{\theta}_0} [\alpha_n M_n M_n'] \left( E_{\boldsymbol{\theta}_0}[M_n M_n'] \right)^{-1}.$$
It can be shown (following the same arguments used in the proof of Theorem \ref{T.FCLT}) that 
\begin{align*}
\frac{1}{n} E_{\boldsymbol{\theta}_0} [\alpha_n M_nM_n'] &\to \alpha(\boldsymbol{\theta}_0) \Sigma , \qquad \text{and} \\
\frac{1}{n} E_{\boldsymbol{\theta}_0} [M_nM_n'] &\to \Sigma ,
\end{align*}
as $n \to \infty$. Therefore, $C_n^* \to \alpha(\boldsymbol{\theta}_0) \Sigma \Sigma^{-1} = \alpha(\boldsymbol{\theta}_0)I$, where $I$ is the identity matrix of $\mathbb{R}^{d\times d}$. We then have that $\alpha(\boldsymbol{\theta}_0)I$ is the asymptotically optimal choice for the control variate coefficient and our new suggested estimator is 
$$Y(C^*_n) = (\alpha_n - \alpha(\boldsymbol{\theta}_0))\nabla L_n(\boldsymbol{\theta}_0).$$
Using again Theorem \ref{T.FCLT} we obtain the following convergence result. 

\bigskip

\begin{prop} \label{P.CLTgoodestimator}
Suppose that Assumption \ref{A.Vuniformity} is satisfied and let $\hat f$ be the solution to Poisson's equation \eqref{eq:PoissonEq} from Lemma \ref{L.Poisson}. Define the functions $g_{ij}$ for $0 \leq i,j \leq d$ according to Theorem~\ref{T.FCLT}. Then,
\begin{equation*} 
Y(C^*_n)  \Rightarrow Z_0 \hat Z \qquad n \to \infty,
\end{equation*}
where $Z = (Z_0, Z_1\dots, Z_d)'$ is a $(d+1)$-dimensional multivariate normal random vector having covariance matrix $\Sigma = (\sigma_{ij})$, where $\sigma_{ij} = \pi(\boldsymbol{\theta}_0) g_{ij}$ for $0 \leq i,j \leq d$, and $\hat Z = (Z_1, \dots, Z_d)'$.
\end{prop}

\begin{proof}
By Theorem \ref{T.FCLT} we have
$$\left( n^{1/2}  (\alpha_n - \alpha(\boldsymbol{\theta}_0)), \,  n^{-1/2} \nabla L_n(\boldsymbol{\theta}_0)'\right) \Rightarrow B(1) \qquad n \to \infty,$$
where $B(t) = (B_0(t), B_1(t), \dots, B_d(t))'$ is a $(d+1)$-dimensional mean zero Brownian motion with covariance matrix $\Sigma = \pi(\boldsymbol{\theta}_0)G$. Let $Z = B(1)$ and $\hat Z = (B_1(1), \dots, B_d(1))$. 

Then, by the continuous mapping principle,
$$(\alpha_n - \alpha(\boldsymbol{\theta}_0)) \nabla L_n(\boldsymbol{\theta}_0)  \Rightarrow Z_0 \hat Z \qquad n \to \infty.$$
\end{proof}

We conclude that $Y(C_n^*)$ has the desired convergence properties and is a suitable estimator for $E_{\boldsymbol{\theta}_0} \left[ \alpha_n \nabla L_n(\boldsymbol{\theta}_0) \right]$. In the next section we consider other alternatives.

\section{An integral-type estimator} \label{S.SecondEstimator}

As mentioned in the introduction, our second proposed estimator is obtained by first deriving an alternative representation for $E_{\boldsymbol{\theta}_0} \left[ \alpha_n \nabla L_n(\boldsymbol{\theta}_0) \right]$ in terms of a discrete stochastic integral. More precisely, we exploit the martingale properties of $\nabla L_n(\boldsymbol{\theta}_0)$ to obtain that:
\begin{align*}
E_{\boldsymbol{\theta}_0} [\alpha_n \nabla L_n(\boldsymbol{\theta}_0)] &= E_{\boldsymbol{\theta}_0} \left[ \frac{1}{n} \sum_{l=0}^{n-1} f(\Phi_l) \sum_{k=1}^n D_k \right] \\
&= \frac{1}{n} \sum_{k=1}^{n} \sum_{l=0}^{k-1} E_{\boldsymbol{\theta}_0}[ f(\Phi_l) E_{\boldsymbol{\theta}_0}[D_k| \mathcal{F}_{k-1}]] +  E_{\boldsymbol{\theta}_0} \left[\frac{1}{n} \sum_{k=1}^{n} \sum_{l=k}^{n-1}  f(\Phi_l) D_k \right]    \\
&= E_{\boldsymbol{\theta}_0} \left[\frac{1}{n} \sum_{k=1}^{n-1} \sum_{l=k}^{n-1}  f(\Phi_l) D_k \right] ,
\end{align*}
where $D_k^i = \frac{\partial}{\partial \theta^i}  p(\boldsymbol{\theta}_0,\Phi_{k-1},\Phi_k)$ and $D_k = (D_k^1, \dots, D_k^d)'$. This suggests using 
\begin{equation} \label{eq:est2}
Y_n \triangleq \frac{1}{n} \sum_{k=1}^{n-1} \sum_{l=k}^{n-1} f(\Phi_l) D_k
\end{equation}
as an estimator for $E_{\boldsymbol{\theta}_0} \left[ \alpha_n \nabla L_n(\boldsymbol{\theta}_0) \right]$. 

Unfortunately, just as the estimator $\alpha_n \nabla L_n(\boldsymbol{\theta}_0)$, $Y_n$ as defined above fails to converge to an a.s. finite random vector. This is a consequence of Theorem \ref{T.FCLT} again.

\begin{prop} \label{P.CLTbadestimator2}
Suppose that Assumption \ref{A.Vuniformity} is satisfied and define for $1 \leq i,j \leq d$ the functions $g_{ij}$ according to Theorem \ref{T.FCLT}. Then,  
$$n^{-1/2} Y_n \Rightarrow \int_0^1 \alpha(\boldsymbol{\theta}_0)(1-s) I dB(s) \qquad n \to \infty$$
in $D([0,1], \mathbb{R}^{d})$, where $B$ is a $d-$dimensional mean zero Brownian motion with covariance matrix $\Sigma = (\sigma_{ij})$, with $\sigma_{ij} = \pi(\boldsymbol{\theta}_0) g_{ij}$ for $1\leq i,j \leq d$, and $I$ is the  identity matrix of $\mathbb{R}^{d\times d}$. 
\end{prop}

\bigskip

As before, we can try to solve the problem of the lack of convergence of $Y_n$ by using a centered estimator of the form
$$Y_n^* \triangleq \frac{1}{n} \sum_{k=1}^{n-1} \sum_{l=k}^{n-1} (f(\Phi_l) - \alpha(\boldsymbol{\theta}_0))D_k.$$
This modification turns out to be the right one, and we obtain the following convergence result for this new estimator.

\bigskip

\begin{prop} \label{P.FCLTgoodestimator}
Suppose that Assumption \ref{A.Vuniformity} is satisfied and define for $0 \leq i,j \leq d$ the functions $g_{ij}$ according to Theorem \ref{T.FCLT}. Then,  
$$Y_n^* \Rightarrow \int_0^1 (B_0(1) - B_0(s)) I d\hat B(s) \qquad n \to \infty$$
in $D([0,1], \mathbb{R}^{d})$, where $B(t) = (B_0(t), B_1(t), \dots, B_d(t))'$ is a $(d+1)$-dimensional mean zero Brownian motion with covariance matrix $\Sigma = (\sigma_{ij})$, where $\sigma_{ij} = \pi(\boldsymbol{\theta}_0) g_{ij}$ for $0\leq i,j \leq d$, $\hat B(t) = (B_1(t), \dots, B_d(t))'$, and $I$ is the identity matrix of $\mathbb{R}^{d\times d}$. 
\end{prop}

\begin{proof}
Let $S_n(t) = \lfloor nt \rfloor (\alpha_{\lfloor nt \rfloor} -  \alpha(\boldsymbol{\theta}_0) ) = \sum_{j=0}^{\lfloor nt \rfloor-1} (f(\Phi_j) - \alpha(\boldsymbol{\theta}_0))$, and note that by Theorem~\ref{T.FCLT} we have
$$n^{-1/2} \left( S_n , \,  \nabla L_{\lfloor n \cdot \rfloor}(\boldsymbol{\theta}_0)' \right) \Rightarrow B' \qquad n \to \infty,$$
in $D([0,1], \mathbb{R}^{d+1})$, where $B(t) = (B_0(t), B_1(t), \dots, B_d(t))'$ is a $(d+1)$-dimensional mean zero Brownian motion with covariance matrix $\Sigma$. Now define the process $\hat W_n(t) = \sum_{j=\lfloor nt \rfloor}^{n-1} (f(\Phi_j) - \alpha(\boldsymbol{\theta}_0))$ with the convention that $\hat W_n(1) \equiv 0$. It follows that $\hat W_n(t) = S_n(1) - S_n(t)$ and the continuous mapping theorem gives
\begin{equation} \label{eq:jointConv}
n^{-1/2} \left(\hat W_n(\cdot), \nabla L_{\lfloor n\cdot \rfloor}(\boldsymbol{\theta}_0)'\right) \Rightarrow (B_0(1) - B_0(\cdot), B_1(\cdot), \dots, B_d(\cdot)) \qquad n \to \infty
\end{equation}
in $D([0,1], \mathbb{R}^{d+1})$. 

Next, define the processes $X_n(t) = n^{-1/2} \hat W_n(t) I$, $X(t) = (B_0(1)-B_0(t))I$, $Z_n(t) = n^{-1/2} \nabla L_{\lfloor nt \rfloor}(\boldsymbol{\theta}_0)$, and $Z(t) = (B_1(t), \dots, B_d(t))'$. Let $ \mathcal{G}_{n,t}= \mathcal{F}_{\lfloor nt \rfloor}$. Clearly, $\{X_n(t): t \in [0,1]\}$ and $\{ Z_n(t): t \in [0,1]\}$ are $\{ \mathcal{G}_{n,t} \}$-adapted and $Z_n(t)$ is a  $\{ \mathcal{G}_{n,t}\}-$martingale. Also, for $t_i = i/n$, 
$$\hat Y_n = \frac{1}{n} \sum_{k=1}^{n-1}\sum_{l=k}^{n-1} (f(\Phi_l)-\alpha(\boldsymbol{\theta}_0)) D_k =  \sum_{k=1}^{n-1} X_n(t_{k-1}) (Z_n(t_k) - Z_n(t_{k-1})) = \int_0^{1-\frac{1}{n}} X_n(s-) \, d Z_n(s).$$
The same steps used in the proof of Proposition \ref{P.CLTbadestimator2} show that the conditions of Theorem 2.7 of \cite{Kurtz_Protter_91} are satisfied, and we obtain that
$$\left(X_n, Z_n, \int_0^{1-\frac{1}{n}} X_n dZ_n\right) \Rightarrow \left(X, Z, \int_0^1X dZ\right) \qquad n \to \infty$$
in $D([0,1], \mathbb{R}^{d\times d} \times \mathbb{R}^d \times \mathbb{R}^d)$. 
\end{proof}

It follows that $Y_n^*$ is a suitable estimator for $E_{\boldsymbol{\theta}_0} \left[ \alpha_n \nabla L_n(\boldsymbol{\theta}_0) \right]$. It remains to compare $Y_n^*$ to $Y(C_n^*)$ from Section \ref{S.FirstEstimator}.

\section{Computation of the asymptotic variance} \label{S.Remarks}

The two previous sections provide details on two potential estimators for $E_{\boldsymbol{\theta}_0} \left[ \alpha_n \nabla L_n(\boldsymbol{\theta}_0) \right]$, namely,
$$Y(C_n^*) = (\alpha_n-\alpha(\boldsymbol{\theta}_0) ) \nabla L_n(\boldsymbol{\theta}_0)$$
and
$$Y_n^* = \frac{1}{n} \sum_{k=1}^{n-1} \sum_{j=k}^{n-1} (f(\Phi_j) - \alpha(\boldsymbol{\theta}_0)) D_k,$$
where $D_k = \nabla p(\boldsymbol{\theta}_0, \Phi_{k-1}, \Phi_k)$. Both of these estimators have the property, under Assumption \ref{A.Vuniformity}, that their expectation converges to $\nabla \alpha(\boldsymbol{\theta}_0)$, i.e.,
$$E_{\boldsymbol{\theta}_0} \left[ Y(C_n^*) \right] \to \nabla \alpha(\boldsymbol{\theta}_0) \qquad \text{and} \qquad E_{\boldsymbol{\theta}_0} \left[ Y_n^* \right] \to \nabla \alpha(\boldsymbol{\theta}_0),$$
as $n \to \infty$ (Theorem \ref{T.Convergence}), and unlike the estimators given in \eqref{eq:FirstBadEstimator} and \eqref{eq:est2}, they converge to a proper limiting distribution (Propositions \ref{P.CLTgoodestimator} and \ref{P.FCLTgoodestimator}). For comparison purposes we compute in this section the variance of these limiting distributions.

First, by Proposition~\ref{P.CLTgoodestimator} we have $Y(C_n^*) \Rightarrow Z_0 \hat Z$, where $Z = (Z_0, Z_1, \dots, Z_d)'$ is a $(d+1)$-dimensional multivariate normal with covariance matrix $\Sigma$ and $\hat Z = (Z_1, \dots, Z_d)'$. Therefore, by Isserlis' theorem, the $(i,j)$th component, $1 \leq i,j \leq d$, of the limiting distribution's covariance matrix is given by 
\begin{align}
\cov_{\boldsymbol{\theta}_0}( Z_0 \hat Z )_{ij} &= E_{\boldsymbol{\theta}_0} [ Z_0^2 Z_i Z_j] - E_{\boldsymbol{\theta}_0} [Z_0 Z_i] E_{\boldsymbol{\theta}_0} [Z_0 Z_j] \notag \\
&= (\sigma_{00} \sigma_{ij} + 2 \sigma_{0i} \sigma_{0j}) - \sigma_{0i} \sigma_{0j} \notag \\
&= \sigma_{00} \sigma_{ij} + \sigma_{0i} \sigma_{0j}. \label{eq:FirstCov}
\end{align}

Similarly, by Proposition \ref{P.FCLTgoodestimator} we have $Y_n^* \Rightarrow \int_0^1 (B_0(1)-B_0(s)) I d \hat B(s)$ in $D([0,1], \mathbb{R}^{d+1})$, where $B(t) = (B_0(t), B_1(t), \dots, B_d(t))'$ is a $(d+1)$-dimensional Brownian motion with covariance matrix $\Sigma$ and $\hat B(t) = (B_1(t), \dots, B_d(t))'$. Since the calculation of the covariance of the limiting distribution in this case is somewhat lengthier, we state the result in the following lemma and postpone the proof to Section \ref{S.Proofs}.

\begin{lemma} \label{L.IntegralCov}
Let $B(t) = (B_0(t), B_1(t), \dots, B_d(t))'$ be a $(d+1)$-dimensional Brownian motion with covariance matrix $\Sigma$. Let $\mathcal{I} = \int_0^1 (B_0(1)-B_0(s)) I d \hat B(s)$, where $I$ is the $\mathbb{R}^{d \times d}$ identity matrix and $\hat B(t) = (B_1(t), \dots, B_d(t))'$. Then, the $(i,j)$th component of the covariance matrix of $\mathcal{I}$ is given by
\begin{equation} \label{eq:SecondCov}
\text{\rm Cov}_{\boldsymbol{\theta}_0}( \mathcal{I})_{ij} = \frac{\sigma_{00} \sigma_{ij}}{2}.
\end{equation}
\end{lemma}

To simplify the notation let 
$$A = \cov_{\boldsymbol{\theta}_0}(Z_0 \hat Z) \qquad \text{and} \qquad B = \cov_{\boldsymbol{\theta}_0} (\mathcal{I})$$
denote the asymptotic covariances of $Y(C_n^*)$ and $Y_n^*$, respectively.  Next define $v = (\sigma_{01}, \sigma_{02}, \dots, \sigma_{0d})'$ and note that \eqref{eq:FirstCov} and \eqref{eq:SecondCov} give
$$A = 2 B + v v'.$$
We now compare the generalized variances of the two estimators, that is, the determinants of their covariance matrices. Provided $B$ is positive definite we obtain
\begin{align*}
\det(A) &= \det(2B + v v') \\
&= \det(2B) \det \left(I + \frac{1}{2} B^{-1} v v' \right) \\
&= \det(2B) \left(1 + \frac{1}{2} v' B^{-1} v \right) \qquad \text{(by Sylvester's determinant theorem)} \\
&= 2^d \det(B) \left(1 + \frac{1}{2} v' B^{-1} v \right),
\end{align*}
where $I$ is the $\mathbb{R}^{d\times d}$ identity matrix. Since $B$ is positive definite, so is $B^{-1}$, and therefore $v' B^{-1} v \geq 0$. We conclude that
$$ \det(A) \geq 2^d \det(B),$$
which suggests that $Y_n^*$ is a better estimator for $\nabla \alpha(\boldsymbol{\theta}_0)$ than $Y(C_n^*)$.

\section{Proofs} \label{S.Proofs}

This last section of the paper contains all the proofs that were not given in the prior sections. The first one corresponds to the martingale properties of $\nabla L_n(\boldsymbol{\theta}_0)$. 

\begin{proof}[Proof of Lemma \ref{L.MG_cond}]
We start by noting that for any $\boldsymbol{\theta} \in \boldsymbol{\Theta}$ and $1 \leq i \leq d$, 
\begin{align*}
\frac{\partial}{\partial \theta_i} \, L_n(\boldsymbol{\theta}) &= \frac{\partial}{\partial \theta_i} \prod_{j=1}^n p(\boldsymbol{\theta}, \Phi_{j-1}, \Phi_j) = \sum_{j=1}^n \frac{L_n(\boldsymbol{\theta})}{p(\boldsymbol{\theta}, \Phi_{j-1}, \Phi_j)}  \cdot \frac{\partial}{\partial \theta_i} p(\boldsymbol{\theta}, \Phi_{j-1}, \Phi_j)  \\
&= L_n(\boldsymbol{\theta}) \sum_{j=1}^n \frac{\partial}{\partial \theta_i} \log p(\boldsymbol{\theta}, \Phi_{j-1}, \Phi_j) .
\end{align*}
Since $L_n(\boldsymbol{\theta}_0) \equiv 1$, it follows that
$$\nabla L_n(\boldsymbol{\theta}_0) = \sum_{j=1}^n D_j.$$

Next, note that for any fixed $x \in {\sf X}$ we have
\begin{align*}
&\left|  \left. \frac{\partial}{\partial \theta_i} \int_{{\sf X}} p(\boldsymbol{\theta}, x, y) P(\boldsymbol{\theta}_0, x, dy) \right|_{\boldsymbol{\theta} = \boldsymbol{\theta}_0} - \int_{{\sf X}}  \frac{\partial}{\partial \theta_i}  p(\boldsymbol{\theta}_0, x, y) P(\boldsymbol{\theta}_0, x,dy) \right| \\
&= \lim_{h \to 0} \left| \int_{{\sf X}} \left( \frac{p(\boldsymbol{\theta}_0 + h{\bf e}_i, x, y) - p(\boldsymbol{\theta}_0, x, y)}{h}  -\frac{\partial}{\partial \theta_i} p(\boldsymbol{\theta}_0, x, y) \right) P(\boldsymbol{\theta}_0, x, dy) \right| \\
&\leq \lim_{h \to 0} \sup_{g:|g| \leq V} \left| \int_{{\sf X}} g(y) \left( \frac{p(\boldsymbol{\theta}_0 + h{\bf e}_i, x, y) - p(\boldsymbol{\theta}_0, x, y)}{h}  -\frac{\partial}{\partial \theta_i} p(\boldsymbol{\theta}_0, x, y) \right) P(\boldsymbol{\theta}_0, x, dy) \right| \\
&\leq V(x) \lim_{h \to 0}  \vertiii{ \frac{P(\boldsymbol{\theta}_0 + h {\bf e}_i) - P(\boldsymbol{\theta}_0)}{h} - P^{(i)}(\boldsymbol{\theta}_0) }_V.
\end{align*}
Therefore,
$$ \left. \frac{\partial}{\partial \theta_i} \int_{{\sf X}} p(\boldsymbol{\theta}, x, y) P(\boldsymbol{\theta}_0, x, dy) \right|_{\boldsymbol{\theta} = \boldsymbol{\theta}_0} = \int_{{\sf X}}  \frac{\partial}{\partial \theta_i}  p(\boldsymbol{\theta}_0, x, y) P(\boldsymbol{\theta}_0, x,dy)$$
for all $x \in {\sf X}$. It follows that
\begin{align*}
E_{\boldsymbol{\theta}_0}[D_k^i | \mathcal{F}_{k-1}] &= E_{\boldsymbol{\theta}_0} \left[E_{\boldsymbol{\theta}_0}\left[  \left.  \frac{\partial}{\partial \theta_i} p(\boldsymbol{\theta}_0, \Phi_{k-1}, \Phi_k) \right| \Phi_{k-1} \right] \right] \\
&= E_{\boldsymbol{\theta}_0} \left[ \int_{{\sf X}}  \frac{\partial}{\partial \theta_i}  p(\boldsymbol{\theta}_0, \Phi_{k-1}, y) P(\boldsymbol{\theta}_0, \Phi_{k-1},dy)  \right] \\
&= E_{\boldsymbol{\theta}_0} \left[  \left. \frac{\partial}{\partial \theta_i} \int_{{\sf X}} p(\boldsymbol{\theta}, \Phi_{k-1}, y) P(\boldsymbol{\theta}_0, \Phi_{k-1}, dy) \right|_{\boldsymbol{\theta} = \boldsymbol{\theta}_0} \right]   \\
&= 0 \hspace{20mm}  \text{(since the integral is equal to one for all $\boldsymbol{\theta}$)},
\end{align*}
which establishes that $M_n \triangleq \nabla L_n(\boldsymbol{\theta}_0)$ is a martingale. To see that it is square integrable let $M_n = (M_n^1, \dots, M_n^d)'$ and note that $E_{\boldsymbol{\theta}_0,x}[(M_n^i)^2] = \sum_{k=1}^n E_{\boldsymbol{\theta}_0}[(D^i_k)^2]$, and
\begin{align*}
E_{\boldsymbol{\theta}_0}[(D_k^i)^2] &= E_{\boldsymbol{\theta}_0}[E_{\boldsymbol{\theta}_0}[(D_k^i)^2 | \mathcal{F}_{k-1}]] \\
&= E_{\boldsymbol{\theta}_0} \left[ \int_{{\sf X}} \left(\frac{\partial}{\partial \theta_i} p(\boldsymbol{\theta}_0, \Phi_{k-1}, y)\right)^2 P(\boldsymbol{\theta}_0, \Phi_{k-1}, dy) \right]  \\
&= E_{\boldsymbol{\theta}_0} [g_{ii}(\Phi_{k-1})], 
\end{align*}
which is finite since $|g_{ii}|_V < \infty$ by Assumption \ref{A.Vuniformity}(vi) and $E_{\boldsymbol{\theta}_0} [V(\Phi_{k-1})] < \infty$. 
\end{proof}

\bigskip

The next proof corresponds to the martingale constructed using the solution to Poisson's equation.

\begin{proof}[Proof of Lemma~\ref{L.Poisson}]
We start by pointing out that by Lemma~\ref{L.ChainProperties}, the chain $\boldsymbol{\Phi}$ is $V$-ergodic for each $\boldsymbol{\theta} \in B_\epsilon(\boldsymbol{\theta}_0)$, and therefore, $\pi(\boldsymbol{\theta}_0) V < \infty$. Also, by Assumption~\ref{A.Vuniformity}(vii), we have that $\pi(\boldsymbol{\theta}_0) f^2 < \infty$.

We now proceed to show the existence of a solution $\hat f$ to Poisson's equation. To this end, note that by Jensen's inequality and Assumption~\ref{A.Vuniformity}(iv) we have
\begin{equation} \label{eq:UseJensen}
P(\boldsymbol{\theta}_0) \sqrt{V(x)} \leq \sqrt{ P(\boldsymbol{\theta}_0) V(x) } \leq  \sqrt{ \lambda V(x) + b 1_K(x)} \leq \sqrt{\lambda V(x)} + \sqrt{b}1_K(x).
\end{equation}
Next, define $\tilde V(x) = (1- \sqrt{\lambda})^{-1} (1 \vee \kappa) \sqrt{V(x)}$, where $\kappa = |f|_{\sqrt{V}}$ and $x \vee y = \max\{ x, y\}$. Using \eqref{eq:UseJensen}  we obtain
\begin{align*}
P(\boldsymbol{\theta}_0) \tilde V(x) &= (1-\sqrt{\lambda})^{-1} (1 \vee \kappa) P(\boldsymbol{\theta}_0) \sqrt{V(x)} \\
&\leq (1-\sqrt{\lambda})^{-1} (1 \vee \kappa) \left( \sqrt{\lambda V(x) }+ \sqrt{b} 1_K(x)   \right)  \\
&= \sqrt{\lambda} \tilde V(x) + (1-\sqrt{\lambda})^{-1} (1 \vee \kappa) \sqrt{b} 1_K(x) \\
&= \tilde V(x) - (1 \vee \kappa) \sqrt{V(x)} + (1-\sqrt{\lambda})^{-1} (1 \vee \kappa) \sqrt{b} 1_K(x).
\end{align*}
It follows that condition (V3) in \cite{Meyn_Tweedie} (see equation (14.16) in \cite{Meyn_Tweedie} or equation (8) in \cite{Glynn_Meyn_96}) is satisfied with $\tilde V$ everywhere finite, $(1 \vee \kappa) \sqrt{V} \geq 1$, and $K$ a small set (hence $K$ petite). Moreover, by Jensen's inequality, 
$$\pi(\boldsymbol{\theta}_0) \tilde V = (1- \sqrt{\lambda})^{-1} (1 \vee \kappa) \pi(\boldsymbol{\theta}_0) \sqrt{V} \leq (1- \sqrt{\lambda})^{-1} (1 \vee \kappa) \sqrt{\pi(\boldsymbol{\theta}_0) V} < \infty.$$
Then, since $|f| \leq (1 \vee \kappa) \sqrt{V}$, Theorem 2.3 in \cite{Glynn_Meyn_96} (Theorem 17.4.2 in \cite{Meyn_Tweedie}) ensures that there exists a solution $\hat f$ to Poisson's equation satisfying $|\hat f| \leq c_0 (\tilde V + 1)$ for some constant $c_0 < \infty$.   This last inequality also implies that $\pi(\boldsymbol{\theta}_0) \hat f^2 < \infty$. Choose $c_1 = 2 c_0 (1-\sqrt{\lambda})^{-1} (1 \vee \kappa)$ to obtain the statement of the lemma.

It remains to show that $Z_n$ is a square-integrable martingale. Clearly,
$$E_{\boldsymbol{\theta}_0} \left[ \hat f(\Phi_k) - P(\boldsymbol{\theta}_0) \hat f (\Phi_{k-1})  \right] = E_{\boldsymbol{\theta}_0} \left[ E_{\boldsymbol{\theta}_0} \left[  \left. \hat f(\Phi_k) \right| \mathcal{F}_{k-1} \right] - P(\boldsymbol{\theta}_0) \hat f (\Phi_{k-1})  \right] = 0,$$
so $Z_n$ is a martingale. To see that it is square-integrable note that
\begin{align*}
E_{\boldsymbol{\theta}_0} \left[ \left( \hat f(\Phi_k) - P(\boldsymbol{\theta}_0) \hat f (\Phi_{k-1})  \right)^2 \right] &= E_{\boldsymbol{\theta}_0} \left[ \left( f(\Phi_k) - \pi(\boldsymbol{\theta}_0)  f   \right)^2 \right] \\
&\leq E_{\boldsymbol{\theta}_0} \left[  f(\Phi_k)^2 \right] +  E_{\boldsymbol{\theta}_0} \left[  |f(\Phi_k)| \right] \pi(\boldsymbol{\theta}_0) f + \left( \pi(\boldsymbol{\theta}_0) f \right)^2.
\end{align*}
Since $|f|_{\sqrt{V}} < \infty$ and both $E_{\boldsymbol{\theta}_0} \left[  V(\Phi_k) \right] < \infty$ and $\pi(\boldsymbol{\theta}_0) V < \infty$, then the above expression is finite, which completes the proof. 
\end{proof}

\bigskip

Next, we give the proof of Theorem \ref{T.Convergence}, which states that under Assumption~\ref{A.Vuniformity} the expectation of $\alpha_n \nabla L_n(\boldsymbol{\theta}_0)$ converges to $\nabla \alpha(\boldsymbol{\theta}_0)$.

\begin{proof}[Proof of Theorem \ref{T.Convergence}]
Define $M_n^i = \sum_{j=1}^n D_j^i$, $1 \leq i \leq d$ as in Lemma \ref{L.MG_cond}, and 
$$Z_n = \sum_{k=1}^n \left( \hat f(\Phi_k) - P(\boldsymbol{\theta}_0) \hat f(\Phi_{k-1})\right), \qquad \xi_k = Z_k - Z_{k-1},$$ as in Lemma \ref{L.Poisson}. By those same lemmas we have that $M_n^i$ and $Z_n$ are square-integrable martingales. 

Next, note that
\begin{align*}
E_{\boldsymbol{\theta}_0} \left[\alpha_n \nabla L_n(\boldsymbol{\theta}_0) \right]  &= E_{\boldsymbol{\theta}_0,x}[\alpha_n M_n^i] \\
&= E_{\boldsymbol{\theta}_0}[(\alpha_n - \alpha(\boldsymbol{\theta}_0)) M_n^i] \\
&= \frac{1}{n} E_{\boldsymbol{\theta}_0}[\hat f(\Phi_0) M_n^i] - \frac{1}{n} E_{\boldsymbol{\theta}_0}[\hat f(\Phi_n) M_n^i] +  \frac{1}{n} \sum_{k=1}^{n} E_{\boldsymbol{\theta}_0}[\xi_k M_n^i] \\
&=\frac{1}{n} \sum_{j=1}^n E_{\boldsymbol{\theta}_0}[\hat f(\Phi_0)  D_j^i] - \frac{1}{n}  E_{\boldsymbol{\theta}_0}[\hat f(\Phi_n)  M_n^i] +  \frac{1}{n} \sum_{k=1}^{n} \sum_{j=1}^n E_{\boldsymbol{\theta}_0}[\xi_k  D_j^i]  \\
&=   - \frac{1}{n} E_{\boldsymbol{\theta}_0}[\hat f(\Phi_n) M_n^i] + \frac{1}{n} \sum_{k=1}^{n} \sum_{j=1}^k E_{\boldsymbol{\theta}_0}[\xi_k D_j^i] \\
&= - \frac{1}{n} E_{\boldsymbol{\theta}_0}[\hat f(\Phi_n) M_n^i] + \frac{1}{n} \sum_{k=1}^{n} E_{\boldsymbol{\theta}_0}[\xi_k D_k^i].
\end{align*} 
To show that $\frac{1}{n} \left| E_{\boldsymbol{\theta}_0}[ \hat f(\Phi_n) M_n^i] \right|  \to 0$ as $n \to \infty$, note that by the Cauchy-Schwarz inequality
\begin{align*}
\frac{1}{n} \left| E_{\boldsymbol{\theta}_0}[ \hat f(\Phi_n) M_n^i] \right|  &\leq \frac{1}{n} \left( E_{\boldsymbol{\theta}_0}[ \hat f(\Phi_n)^2]  \right)^{1/2} \left( E_{\boldsymbol{\theta}_0}[ (M_n^i)^2 ]\right)^{1/2} \\
&= \left( \frac{1}{n} E_{\boldsymbol{\theta}_0}[ \hat f(\Phi_n)^2]  \right)^{1/2} \left( \frac{1}{n} \sum_{j=1}^n E_{\boldsymbol{\theta}_0}[ (D_j^i)^2 ]\right)^{1/2}.
\end{align*}
Also, by Lemma \ref{L.Poisson} we have $\hat f^2 \leq c_1 V$, and since by Lemma \ref{L.ChainProperties} $\boldsymbol{\Phi}$ is $V$-ergodic, we obtain that $E_{\boldsymbol{\theta}_0}[ \hat f(\Phi_n)^2] \to \pi(\boldsymbol{\theta}_0) \hat f^2 < \infty$ as $n \to \infty$. This in turn implies that $\frac{1}{n} E_{\boldsymbol{\theta}_0}[ \hat f(\Phi_n)^2] \to 0$ as $n \to \infty$. For the other term we have by Assumption \ref{A.Vuniformity}(vi) that $|g_{ii}|_V < \infty$, and therefore $E_{\boldsymbol{\theta}_0}[ g_{ii}(\Phi_n)] \to \pi(\boldsymbol{\theta}_0) g_{ii} < \infty$. Hence,
$$\lim_{n \to \infty} \frac{1}{n} \sum_{j=1}^n E_{\boldsymbol{\theta}_0} [ (D_j^i)^2] = \lim_{n \to \infty} \frac{1}{n} \sum_{j=1}^n E_{\boldsymbol{\theta}_0} [ g_{ii}(\Phi_{j-1}) ] = \pi(\boldsymbol{\theta}_0) g_{ii}.$$
We conclude that $\frac{1}{n} \left| E_{\boldsymbol{\theta}_0}[ \hat f(X_n) M_n^i] \right|  \to 0$ as $n \to \infty$. 

To show that $\frac{1}{n} \sum_{k=1}^{n} E_{\boldsymbol{\theta}_0}[\xi_k D_k^i] \to \frac{\partial}{\partial \theta_i} \alpha (\boldsymbol{\theta}_0)$ note that
\begin{align*}
E_{\boldsymbol{\theta}_0}[\xi_k D_k] &= E_{\boldsymbol{\theta}_0}[E_{\boldsymbol{\theta}_0}[\xi_k D_k| \mathcal{F}_{k-1}]] \\
&= E_{\boldsymbol{\theta}_0}\left[ \int_{{\sf X}} \left( \hat f(y) - P(\boldsymbol{\theta}_0) \hat f(\Phi_{k-1}) \right) \frac{\partial}{\partial \theta_i} p(\boldsymbol{\theta}_0, \Phi_{k-1}, y) P(\boldsymbol{\theta}_0, \Phi_{k-1}, dy)    \right] \\
&= E_{\boldsymbol{\theta}_0}\left[ \int_{{\sf X}} \hat f(y) \frac{\partial}{\partial \theta_i} p(\boldsymbol{\theta}_0, \Phi_{k-1}, y) P(\boldsymbol{\theta}_0, \Phi_{k-1}, dy)    \right]. 
\end{align*}
Let $h_i(x) = \int_{{\sf X}} \hat f(y) \frac{\partial}{\partial \theta_i} p(\boldsymbol{\theta}_0, x, y) P(\boldsymbol{\theta}_0, x, dy)$ and note that by the Cauchy-Schwarz inequality
\begin{align*}
|h_i(x)| &\leq \left( \int_{{\sf X}} \hat f^2(y)  P(\boldsymbol{\theta}_0, x, dy) \right)^{1/2} \left( \int_{{\sf X}} \left( \frac{\partial}{\partial \theta_i} p(\boldsymbol{\theta}_0, x, y) \right)^2 P(\boldsymbol{\theta}_0, x, dy) \right)^{1/2} \\
&= \left( P(\boldsymbol{\theta}_0) \hat f^2(x)  \right)^{1/2} \left( g_{ii}(x) \right)^{1/2} \\
&\leq c_1 \left(  P(\boldsymbol{\theta}_0) V(x)  \right)^{1/2} \left(  |g_{ii}|_V V(x) \right)^{1/2} \\
&\leq c_1 \left(  |g_{ii}|_V V(x) (  \lambda V(x) + b ) \right)^{1/2} \\
&\leq c_1 \left( |g_{ii}|_V ( \lambda + b ) \right)^{1/2} V(x),
\end{align*}
and therefore $|h_i|_V < \infty$. It follows from the same arguments used above that 
$$\lim_{n \to \infty} \frac{1}{n} \sum_{k=1}^{n} E_{\boldsymbol{\theta}_0}[\xi_k D_k^i] = \lim_{n \to \infty} \frac{1}{n} \sum_{k=1}^{n} E_{\boldsymbol{\theta}_0}[ h_i(\Phi_{k-1}) ] = \pi(\boldsymbol{\theta}_0) h_i,$$
with the limit $\pi(\boldsymbol{\theta}_0) h_i$ well-defined and finite. It only remains to show that $\pi(\boldsymbol{\theta}_0) h_i = \frac{\partial}{\partial \theta_i} \alpha (\boldsymbol{\theta}_0)$. To do this first note that
\begin{align*}
\frac{\partial}{\partial \theta_i} \alpha (\boldsymbol{\theta}_0) &= \left. \frac{\partial}{\partial \theta_i} \int_{{\sf X}}  f(x) \pi(\boldsymbol{\theta}, dx) \right|_{\boldsymbol{\theta} = \boldsymbol{\theta}_0} \\
&= \left. \frac{\partial}{\partial \theta_i} \int_{{\sf X}}  \left( \hat f(x) - P(\boldsymbol{\theta}_0) \hat f(x) + \alpha(\boldsymbol{\theta}_0) \right)\pi(\boldsymbol{\theta}, dx) \right|_{\boldsymbol{\theta} = \boldsymbol{\theta}_0} \\
&= \left. \frac{\partial}{\partial \theta_i} \int_{{\sf X}}  \left( \hat f(x) - P(\boldsymbol{\theta}_0) \hat f(x)  \right)\pi(\boldsymbol{\theta}, dx) \right|_{\boldsymbol{\theta} = \boldsymbol{\theta}_0} \\
&= \lim_{h \to 0} \int_{{\sf X}}   \left( \hat f(x) - P(\boldsymbol{\theta}_0) \hat f(x)  \right)  \frac{\pi(\boldsymbol{\theta}_0+h{\bf e}_i, dx) - \pi(\boldsymbol{\theta}_0, dx)}{h}   \\
&= \lim_{h \to 0} \int_{{\sf X}}   \left( \hat f(x) - P(\boldsymbol{\theta}_0) \hat f(x)  \right)  \frac{\pi(\boldsymbol{\theta}_0+h{\bf e}_i, dx) }{h}  \\
&=  \lim_{h \to 0} \int_{{\sf X}}   \left( \hat f(x) -P(\boldsymbol{\theta}_0+h{\bf e}_i) \hat f(x) + P(\boldsymbol{\theta}_0+h{\bf e}_i) \hat f(x) - P(\boldsymbol{\theta}_0) \hat f(x)  \right)  \frac{\pi(\boldsymbol{\theta}_0+h{\bf e}_i, dx) }{h} \\
&= \lim_{h \to 0} \int_{{\sf X}}   \left( \frac{ P(\boldsymbol{\theta}_0+h{\bf e}_i) \hat f(x) - P(\boldsymbol{\theta}_0) \hat f(x) }{h}  \right)  \pi(\boldsymbol{\theta}_0+h{\bf e}_i, dx) ,
\end{align*}
where in the fifth and seventh steps we used the identity $\pi(\boldsymbol{\theta}) P(\boldsymbol{\theta}) = \pi(\boldsymbol{\theta})$ for all $\boldsymbol{\theta} \in B_\epsilon(\boldsymbol{\theta}_0)$. Next, note that $h_i(x) = P^{(i)}(\boldsymbol{\theta}_0) \hat f(x)$, from where it follows that
\begin{align}
&\left| \pi(\boldsymbol{\theta}_0)h_i - \frac{\partial}{\partial \theta_i} \alpha (\boldsymbol{\theta}_0) \right| \notag \\
&= \left| \int_{{\sf X}} P^{(i)}(\boldsymbol{\theta}_0) \hat f(x)   \pi(\boldsymbol{\theta}_0, dx)  -  \lim_{h \to 0} \int_{{\sf X}}   \left( \frac{ P(\boldsymbol{\theta}_0+h{\bf e}_i) \hat f(x) - P(\boldsymbol{\theta}_0) \hat f(x) }{h}  \right)  \pi(\boldsymbol{\theta}_0+h{\bf e}_i, dx) \right| \notag \\
&\leq  \lim_{h \to 0} \left| \int_{{\sf X}} P^{(i)}(\boldsymbol{\theta}_0) \hat f(x)  \left( \pi(\boldsymbol{\theta}_0,dx) -  \pi(\boldsymbol{\theta}_0 + h {\bf e}_i, dx) \right) \right| \label{eq:dominatedConv} \\
&\hspace{5mm} +  \lim_{h \to 0} \left|  \int_{{\sf X}}   \left( P^{(i)}(\boldsymbol{\theta}_0) \hat f(x)  - \frac{ P(\boldsymbol{\theta}_0+h{\bf e}_i) \hat f(x) - P(\boldsymbol{\theta}_0) \hat f(x) }{h}  \right) \pi(\boldsymbol{\theta}_0+h{\bf e}_i, dx)  \right| . \label{eq:VnormConv}
\end{align}
It remains to show that the last two limits are zero. To analyze \eqref{eq:dominatedConv} recall that $|h_i|_V < \infty$, from where it follows that \eqref{eq:dominatedConv} is bounded by
\begin{align*}
\lim_{h \to 0} \left|\left| \pi(\boldsymbol{\theta}_0) - \pi(\boldsymbol{\theta}_0+h{\bf e}_i) \right|\right|_V = 0 \qquad \text{(by Lemma \ref{L.ChainProperties})}.
\end{align*}
And to show that \eqref{eq:VnormConv} is zero as well note that
\begin{align*}
&\left| P^{(i)}(\boldsymbol{\theta}_0) \hat f(x)   -\frac{ P(\boldsymbol{\theta}_0+h{\bf e}_i) \hat f(x) - P(\boldsymbol{\theta}_0) \hat f(x) }{h}  \right| \\
&\leq |\hat f|_V \left|\left| P^{(i)}(\boldsymbol{\theta}_0, x, \cdot) - \frac{P(\boldsymbol{\theta}_0 + h {\bf e}_i, x, \cdot)  - P(\boldsymbol{\theta}_0, x, \cdot ) }{h} \right|\right|_V \\
&\leq |\hat f|_V V(x) \vertiii{  P^{(i)}(\boldsymbol{\theta}_0) - \frac{P(\boldsymbol{\theta}_0 + h {\bf e}_i)  - P(\boldsymbol{\theta}_0) }{h} }_V,
\end{align*}
which combined with $\pi(\boldsymbol{\theta}_0)V < \infty$ gives that \eqref{eq:VnormConv} is bounded by
\begin{align*}
&\lim_{h \to 0} \int_{{\sf X}}  |\hat f|_V V(x) \vertiii{  P^{(i)}(\boldsymbol{\theta}_0) - \frac{P(\boldsymbol{\theta}_0 + h {\bf e}_i)  - P(\boldsymbol{\theta}_0) }{h} }_V \pi(\boldsymbol{\theta}_0 + h{\bf e}_i, dx) \\
&= |\hat f|_V   \lim_{h \to 0} \vertiii{ P^{(i)}(\boldsymbol{\theta}_0) - \frac{P(\boldsymbol{\theta}_0 + h {\bf e}_i)  - P(\boldsymbol{\theta}_0) }{h} }_V  \pi(\boldsymbol{\theta}_0+ h{\bf e}_i) V \\
&\leq |\hat f|_V   \lim_{h \to 0} \vertiii{ P^{(i)}(\boldsymbol{\theta}_0) - \frac{P(\boldsymbol{\theta}_0 + h {\bf e}_i)  - P(\boldsymbol{\theta}_0) }{h} }_V  \left( \pi(\boldsymbol{\theta}_0) V + \left|\left| \pi(\boldsymbol{\theta}_0+ h{\bf e}_i) - \pi(\boldsymbol{\theta}_0) \right|\right|_V \right) = 0.
\end{align*}
This completes the proof. 
\end{proof}

The following is the proof of the main weak convergence theorem that is used to describe the behavior of all four estimators considered in Sections \ref{S.FirstEstimator} and \ref{S.SecondEstimator}. It is essentially an application of the Functional Central Limit Theorem for multivariate martingales found in \cite{Whitt_07} (see also Theorems 1.4 and 1.2 in Chapter 7 of \cite{Ethier_Kurtz}).

\begin{proof}[Proof of Theorem \ref{T.FCLT}]
For $m \in \mathbb{N}$ let $Z_m = \sum_{k=1}^m \left( \hat f(\Phi_k) - P(\boldsymbol{\theta}_0) \hat f(\Phi_{k-1}) \right)$ and $M_m = \nabla L_m(\boldsymbol{\theta}_0)$. Next, define the process $X_n(t) = n^{-1/2} \left( Z_{\lfloor nt \rfloor}, M_{\lfloor nt \rfloor}' \right)'$ and the filtrations $\mathcal{G}_{n,t}= \mathcal{F}_{\lfloor nt \rfloor} = \sigma( \Phi_0, \dots, \Phi_{\lfloor n t \rfloor})$. Note that by Lemmas \ref{L.MG_cond} and \ref{L.Poisson} $\{ X_n(t) : t\in [0,1]\}$ is a square integrable martingale with respect to $\mathcal{G}_{n,t}$. Moreover, by \eqref{eq:alphas} we have
$$\left( n^{-1/2} \lfloor nt \rfloor (\alpha_{\lfloor nt \rfloor} - \alpha(\boldsymbol{\theta}_0)), \, n^{-1/2} \nabla L_{\lfloor nt \rfloor}(\boldsymbol{\theta}_0)' \right) = \left(   n^{-1/2} ( \hat f(\Phi_0) - \hat f(\Phi_{\lfloor n t \rfloor}) ), \, {\bf 0}' \right) + X_n(t),$$
where ${\bf 0}$ is the zero vector in $\mathbb{R}^d$. Note that
\begin{align*}
\sup_{0 \leq t \leq 1} n^{-1/2} |\hat f(\Phi_0) - \hat f(\Phi_{\lfloor n t \rfloor}) )| \leq 2 \max_{0 \leq k \leq n} \frac{|\hat f(\boldsymbol{\Phi}_k)|}{n^{1/2}}.
\end{align*}
Since $\boldsymbol{\Phi}$ is $V$-ergodic and $|\hat f^2|_V < \infty$, Theorem~17.3.3 in \cite{Meyn_Tweedie} gives
$$\max_{1 \leq k \leq n} \frac{(\hat  f(\boldsymbol{\Phi}_k))^2}{n} \to 0 \qquad \text{a.s. } P(\boldsymbol{\theta}),$$
which in turn implies that $\left(   n^{-1/2} ( \hat f(\Phi_0) - \hat f(\Phi_{\lfloor n t \rfloor}) ), \, {\bf 0}' \right) \Rightarrow {\bf 0}'$ in $D([0,1], \mathbb{R}^{d+1})$. It follows by Slutsky's lemma that it suffices to show that $X_n \Rightarrow B$ in $D([0,1], \mathbb{R}^{d+1})$. We will do so by showing that $X_n$ satisfies condition (ii) of Theorem 2.1 in \cite{Whitt_07}. 
 
Let $\xi_k^n = X_n(k/n) - X_n((k-1)/n)$ and consider the matrix $A_n = A_n(t) \in \mathbb{R}^{(d+1)\times (d+1)}$ whose $(i,j)$th component is given by
$$A_n^{ij}(t) = \sum_{k=1}^{\lfloor n t \rfloor} E_{\boldsymbol{\theta}_0} \left[ \left. (\xi_k^n)_i (\xi_k^n)_j \right| \mathcal{F}_{k-1} \right] . $$
Then $X_n^i(t) X_n^j(t) - A_n^{ij}(t)$ is a martingale adapted to $\mathcal{G}_{n,t}$ for each $0 \leq i,j \leq d$, and therefore, the $A_n^{ij} = \langle X_n^i, X_n^j \rangle$ are the predictable quadratic-covariation processes of $X_n$. Also, for $1 \leq i,j \leq d$ we have
\begin{align*}
A_n^{ij}(t) &= \frac{1}{n} \sum_{k=1}^{\lfloor nt \rfloor} E_{\boldsymbol{\theta}_0} \left[ \left. \frac{\partial}{\partial \theta_i}  p(\boldsymbol{\theta}_0, \Phi_{k-1}, \Phi_k) \, \frac{\partial}{\partial \theta_j}  p(\boldsymbol{\theta}_0, \Phi_{k-1}, \Phi_k) \right| \mathcal{F}_{k-1} \right] \\
&= \frac{1}{n} \sum_{k=1}^{\lfloor nt \rfloor} \int_{\sf X} \frac{\partial}{\partial \theta_i} p(\boldsymbol{\theta}_0, \Phi_{k-1}, y) \, \frac{\partial}{\partial \theta_j} p(\boldsymbol{\theta}_0, \Phi_{k-1}, y) P(\boldsymbol{\theta}_0,\Phi_{k-1}, dy) \\
&= \frac{1}{n} \sum_{k=1}^{\lfloor nt \rfloor} g_{ij}(\Phi_{k-1}),
\end{align*}
and for $1 \leq i \leq d$, 
\begin{align*}
A_n^{0i}(t) = A_n^{i0}(t) &= \frac{1}{n} \sum_{k=1}^{\lfloor nt \rfloor} E_{\boldsymbol{\theta}_0} \left[ \left. \frac{\partial}{\partial \theta_i}  p(\boldsymbol{\theta}_0, \Phi_{k-1}, \Phi_k)  \left( \hat f(\Phi_k) - P(\boldsymbol{\theta}_0) \hat f(\Phi_{k-1}) \right) \right| \mathcal{F}_{k-1} \right] \\
&= \frac{1}{n} \sum_{k=1}^{\lfloor nt \rfloor} \int_{\sf X} \frac{\partial}{\partial \theta_i} p(\boldsymbol{\theta}_0, \Phi_{k-1}, y) \hat f(y)  P(\boldsymbol{\theta}_0,\Phi_{k-1}, dy) \\
&= \frac{1}{n} \sum_{k=1}^{\lfloor nt \rfloor} g_{0i}(\Phi_{k-1}).
\end{align*} 
Similarly,
$$A_n^{00}(t) = \frac{1}{n}   \sum_{k=1}^{\lfloor nt \rfloor} g_{00}(\Phi_{k-1}).$$
Since $\pi(\boldsymbol{\theta}_0)|g_{ij}| \leq \pi(\boldsymbol{\theta}_0) (g_{ii} g_{jj})^{1/2} \leq (\pi(\boldsymbol{\theta}_0)g_{ii})^{1/2} (\pi(\boldsymbol{\theta}_0)g_{jj})^{1/2} < \infty$ for each $0\leq i,j \leq d$, we have by Theorem 17.0.1 in \cite{Meyn_Tweedie} that
$$\frac{\lfloor nt \rfloor}{n} \cdot \frac{1}{\lfloor nt \rfloor} \sum_{k=1}^{\lfloor nt \rfloor} g_{ij}(\Phi_{k-1}) \to t\pi(\boldsymbol{\theta}_0) g_{ij} \qquad \text{a.s. } P(\boldsymbol{\theta}).$$

Also, for each $0 \leq i,j \leq d$,
\begin{align*}
\lim_{n \to \infty} E_{\boldsymbol{\theta}_0} \left[ \sup_{t \in [0,1]} | A_n^{ij}(t) - A_n^{ij}(t-) | \right] &= \lim_{n \to \infty}  E_{\boldsymbol{\theta}_0} \left[ \max_{1 \leq k \leq n} \left| E_{\boldsymbol{\theta}_0} \left[ \left.  (\xi_k^n)_i (\xi_k^n)_j \right| \mathcal{F}_{k-1} \right] \right| \right] \\
&= \lim_{n \to \infty} \frac{1}{n} E_{\boldsymbol{\theta}_0} \left[ \max_{1 \leq k \leq n} \left|   g_{ij}(\Phi_{k-1})  \right| \right] \\
&\leq |g_{ij}|_V \lim_{n \to \infty} \frac{1}{n} E_{\boldsymbol{\theta}_0} \left[ \max_{1 \leq k \leq n}    V(\Phi_{k-1})   \right] \\
&\leq |g_{ij}|_V \lim_{n \to \infty} \frac{1}{n} E_{\boldsymbol{\theta}_0} \left[ \max_{1 \leq k \leq n} \left( \sqrt{n} +   V(\Phi_{k-1}) 1(V(\Phi_{k-1}) > \sqrt{n})   \right) \right] \\
&\leq  |g_{ij}|_V \lim_{n \to \infty} \frac{1}{n} \sum_{k=1}^n E_{\boldsymbol{\theta}_0} \left[    V(\Phi_{k-1}) 1(V(\Phi_{k-1}) > \sqrt{n})  \right] .
\end{align*}
To see that the last expression converges to zero let $h_m(x) = V(x) 1(V(x) > m)$, and note that $E_{\boldsymbol{\theta}_0} [h_m(\Phi_k)] \to \pi(\boldsymbol{\theta}_0) h_m < \infty$ as $k \to \infty$, and monotone convergence gives $\pi(\boldsymbol{\theta}_0) h_m \to 0$ as $m \to \infty$, therefore we can choose $N \in \mathbb{N}$ large enough so that $\pi(\boldsymbol{\theta}_0) h_N < \delta$. It follows that
$$\lim_{n \to \infty} \frac{1}{n} \sum_{k=1}^n E_{\boldsymbol{\theta}_0} \left[ h_{\sqrt{n}}(\Phi_{k-1}) \right]  \leq \lim_{n \to \infty} \frac{1}{n} \sum_{k=1}^n E_{\boldsymbol{\theta}_0} \left[ h_N(\Phi_{k-1}) \right] = \pi(\boldsymbol{\theta}_0) h_N < \delta,$$
and since $\delta > 0$ was arbitrary, the limit is zero. 

For a vector ${\bf x} = (x_0, x_1, \dots, x_d)' \in \mathbb{R}^{d+1}$ let $|{\bf x}| = \left( \sum_{i=0}^{d} x_i^2 \right)^{1/2}$. Then, by similar arguments as those used above, 
\begin{align*}
\lim_{n \to \infty} E_{\boldsymbol{\theta}_0} \left[ \sup_{t \in [0,1]} | X_n(t) - X_n(t-) |^2 \right] &=  \lim_{n \to \infty} E_{\boldsymbol{\theta}_0} \left[ \max_{1 \leq k \leq n} \left| \xi_k^n  \right|^2 \right] \\
&= \sum_{i=0}^d \lim_{n \to \infty} E_{\boldsymbol{\theta}_0} \left[ \max_{1 \leq k \leq n}  (\xi_k^n)^2_i \right]  \\
&\leq \sum_{i=0}^d \lim_{n \to \infty} \sum_{k=1}^n  E_{\boldsymbol{\theta}_0} \left[ (\xi_k^n)^2_i 1( (\xi_k^n)^2_i > n^{-1/2})   \right] \\
&= \sum_{i=0}^d \lim_{n \to \infty} \frac{1}{n} \sum_{k=1}^n E_{\boldsymbol{\theta}_0} \left[ \hat h_{i,\sqrt{n}}(\Phi_{k-1}) \right],
\end{align*}
where
\begin{align*}
\hat h_{i,m}(x) &= \int_{\sf X} \left( \frac{\partial}{\partial \theta_i} p(\boldsymbol{\theta}_0, y, x) \right)^2 1\left(  \left( \frac{\partial}{\partial \theta_i} p(\boldsymbol{\theta}_0, y, x) \right)^2 > m \right) P(\boldsymbol{\theta}_0, x,dy), \qquad 1 \leq i \leq d, \\
\hat h_{0,m}(x) &= \int_{\sf X} \left( \hat f(y) - P(\boldsymbol{\theta}_0)\hat f(x) \right)^2 1\left(  \left( \hat f(y) - P(\boldsymbol{\theta}_0)\hat f(x) \right)^2 > m \right) P(\boldsymbol{\theta}_0, x,dy).
\end{align*}
Since we have that $\hat h_{i,m} \leq |g_{ii}|_V V$ for all $0 \leq i \leq d$, then $E_{\boldsymbol{\theta}_0} [\hat h_{i,m} (\Phi_{k-1})] \to \pi(\boldsymbol{\theta}_0) \hat h_{i,m} < \infty$ for each fixed $m \in \mathbb{N}$ and the same arguments used before give 
$$\lim_{n \to \infty} \frac{1}{n} \sum_{k=1}^n E_{\boldsymbol{\theta}_0} \left[ \hat h_{i,\sqrt{n}}(\Phi_{k-1}) \right] = 0.$$
It now follows from Theorem 2.1 in \cite{Whitt_07} that $X_n \Rightarrow B$ in $D([0,1], \mathbb{R}^{d+1})$. 
\end{proof}

\bigskip

The next proof corresponds to Proposition \ref{P.CLTbadestimator2}, which shows the lack of convergence of the discrete integral 
$$\frac{1}{n} \sum_{k=1}^{n-1} \sum_{l=k}^{n-1} f(\boldsymbol{\Phi}_l) D_k.$$

\begin{proof}[Proof of Proposition \ref{P.CLTbadestimator2}]
First note that the $V$-ergodicity of $\boldsymbol{\Phi}$, the observation that $\pi(\boldsymbol{\theta}_0)|f| < \infty$, and Theorem~17.0.1 in \cite{Meyn_Tweedie}, gives for any $t \in [0,1]$, 
$$W_n(t) = \frac{1}{n} \sum_{l=0}^{n-1} f(\boldsymbol{\Phi}_l) - \frac{\lfloor nt \rfloor + 1}{n} \cdot \frac{1}{\lfloor nt \rfloor + 1} \sum_{l = 0}^{\lfloor nt \rfloor } f(\boldsymbol{\Phi}_l) \to \alpha(\boldsymbol{\theta}_0)(1-t) = W(t) \qquad \text{a.s. } P(\boldsymbol{\theta}_0),$$
as $n \to \infty$. Moreover,
\begin{align*}
n (W_n(t) - W(t)) &= \sum_{l= \lfloor nt \rfloor + 1}^{n-1} (f(\boldsymbol{\Phi}_l) - \alpha(\boldsymbol{\theta}_0)) + (nt-\lfloor nt \rfloor - 1) \alpha(\boldsymbol{\theta}_0) \\
&= S_n(1) - S_n(t) + (nt-\lfloor nt \rfloor - 1) \alpha(\boldsymbol{\theta}_0),
\end{align*}
where $S_n(t) = \sum_{j=0}^{\lfloor nt \rfloor -1} (f(\boldsymbol{\Phi}_j) - \alpha(\boldsymbol{\theta}_0))$ and $n^{-1/2} S_n \Rightarrow B_0$ in $D([0,1], \mathbb{R})$ by Theorem~\ref{T.FCLT}, with $B_0$ a mean zero Brownian motion.  It follows that $W_n(t) - W(t) \Rightarrow 0$ in $D([0,1], \mathbb{R})$.  Also, by Theorem~\ref{T.FCLT} again we have that  $n^{-1/2} \nabla L_{\lfloor n \cdot \rfloor}(\boldsymbol{\theta}_0) \Rightarrow B$ in $D([0,1], \mathbb{R}^d)$, where $B$ is a zero-mean $d-$dimensional Brownian motion with covariance matrix $\Sigma$. 

It follows that since $W$ is a non-random element of $D([0,1], \mathbb{R})$, 
\begin{equation} \label{eq:Joint1}
\left( W_n, \, n^{-1/2} \nabla L_{\lfloor n \cdot \rfloor}(\boldsymbol{\theta}_0)' \right) \Rightarrow (W, B') \qquad n \to \infty
\end{equation}
in $D([0,1], \mathbb{R}^{d+1})$. 

Next, define the processes $X_n(t) = W_n(t) I$, $X(t) = W(t)I$, $Z_n(t) = n^{-1/2} \nabla L_{\lfloor nt \rfloor}(\boldsymbol{\theta}_0)$, and $Z(t) = B(t)$, where $I$ is the identity matrix of $\mathbb{R}^{d\times d}$. Define $\mathcal{G}_{n,t}= \mathcal{F}_{\lfloor nt \rfloor}$. Clearly, $X_n$ and $Z_n$ are $\{\mathcal{G}_{n,t}\}$-adapted and $Z_n$ is a  $\{ \mathcal{G}_{n,t}\}$-martingale. Also, for $t_i = i/n$, 
$$n^{-1/2} Y_n = \frac{1}{n^{1/2}} \frac{1}{n} \sum_{k=1}^{n-1} \sum_{l=k}^{n-1} f(\Phi_l) D_k =  \sum_{k=1}^{n-1} X_n(t_{k-1}) (Z_n(t_k) - Z_n(t_{k-1})) = \int_0^{1-\frac{1}{n}} X_n(s-) \, d Z_n(s).$$
Consider now the process
$$[Z_n]_t = \frac{1}{n} \sum_{k=1}^{\lfloor nt \rfloor} D_kD_k'$$
and note that the $(i,j)^{th}$ element of $\left| E_{\boldsymbol{\theta}_0} [[Z_n]_t]\right|$ ($1 \leq i, j \leq d$) is 
\begin{align*}
\left|\frac{1}{n} \sum_{k=1}^{\lfloor nt \rfloor} E_{\boldsymbol{\theta}_0} [D_k^i D_k^j] \right| &= \left|\frac{1}{n} \sum_{k=1}^{\lfloor nt \rfloor} E_{\boldsymbol{\theta}_0} \left[ E_{\boldsymbol{\theta}_0}[D_k^i D_k^j| \mathcal{F}_{k-1}]  \right] \right| = \left| E_{\theta_0} \left[ \frac{1}{n} \sum_{k=1}^{\lfloor nt \rfloor} g_{ij}(\Phi_{k-1})\right] \right| \\
&\leq t \sup_n E_{\boldsymbol{\theta}_0} \left[ \left| \frac{1}{n} \sum_{k=1}^{n} g_{ij}(\Phi_{k-1}) \right|\right]  .
\end{align*}
For each $\alpha > 0$ let $\tau_n^\alpha = 2\alpha$ and note that $P_{\boldsymbol{\theta}_0}(\tau_n^\alpha \leq \alpha) = 0 \leq 1/\alpha$ and
$$\sup_n E_{\boldsymbol{\theta}_0} [ [Z_n]_{t \wedge \tau_n^\alpha}] \leq \sup_n E_{\boldsymbol{\theta}_0}[[Z_n]_{2\alpha}] \leq 2\alpha \max_{1 \leq i,j \leq d} \sup_n E_{\boldsymbol{\theta}_0} \left[ \left| \frac{1}{n} \sum_{k=1}^{n} g_{ij}(\Phi_{k-1}) \right|\right] <  \infty.$$
Finally, by \eqref{eq:Joint1} we have $(X_n, Z_n) \Rightarrow (X, Z)$ in $D([0,1], \mathbb{R}^{d\times d} \times \mathbb{R}^d)$. 
Therefore, the conditions of Theorem~2.7 of \cite{Kurtz_Protter_91} are satisfied and we have
$$\left(X_n, Z_n, \int_0^{1-\frac{1}{n}} X_n dZ_n\right) \Rightarrow \left(X, Z, \int_0^1 X dZ\right) \qquad n \to \infty$$
in $D([0,1], \mathbb{R}^{d\times d} \times \mathbb{R}^d \times \mathbb{R}^d)$.   
\end{proof}

The last proof in the paper corresponds to the calculation of the variance of $\int_0^1 (B_0(1) - B_0(s)) I d\hat B(s)$. 

\begin{proof}[Proof of Lemma \ref{L.IntegralCov}]
Note that we can write the limit as $B_0(1) \hat B(1) - \int_0^1 B_0(s) I d\hat B(s)$. Let $U = B_0(1) \hat B(1)$ and $V = \int_0^1 B_0(s) I d\hat B(s)$, then, since $E_{\boldsymbol{\theta}_0}[ V] = 0$, the covariance matrix of the limiting distribution is given by
\begin{align*}
\cov_{\boldsymbol{\theta}_0} \left( \mathcal{I} \right) &= E_{\boldsymbol{\theta}_0} [ (U - V - E_{\boldsymbol{\theta}_0}[U]) (U - V - E_{\boldsymbol{\theta}_0}[U])' ] \\
&= E_{\boldsymbol{\theta}_0} [ U U']   - E_{\boldsymbol{\theta}_0} [  U V'] - E_{\boldsymbol{\theta}_0} [ U  ] E_{\boldsymbol{\theta}_0}[U]' - E_{\boldsymbol{\theta}_0} [ V U'] + E_{\boldsymbol{\theta}_0} [ V V']  \\
&= \cov_{\boldsymbol{\theta}_0} ( U ) - E_{\boldsymbol{\theta}_0} [  U V'] - \left( E_{\boldsymbol{\theta}_0} [ U V'] \right)' + E_{\boldsymbol{\theta}_0} [ V V']. 
\end{align*}
Note that $U \stackrel{\mathcal{D}}{=} Z_0 \hat Z$, i.e., the limiting distribution of $Y(C_n^*)$, so the $(i,j)$th component of $\cov_{\boldsymbol{\theta}_0} ( U )$ is $\sigma_{00} \sigma_{ij} + \sigma_{0i} \sigma_{0j}$. To compute the remaining expectations let $W(t) = (W_0(t), W_1(t), \dots, W_d(t))'$ be a standard $(d+1)$-dimensional Brownian motion and write $\Sigma = C C'$, where $C = (c_{ij}) \in \mathbb{R}^{(d+1)\times(d+1)}$. Then, we can rewrite
$$V = \int_0^1 B_0(s) C dW(s), \qquad B(t) = \int_0^t C dW(s),$$
and
\begin{align*}
\left( E_{\boldsymbol{\theta}_0} [ V V'] \right)_{ij} &= E_{\boldsymbol{\theta}_0} \left[ V_i V_j  \right] = E_{\boldsymbol{\theta}_0} \left[  \int_0^1 B_0(s) \sum_{k=0}^d c_{ik} d W_k(s)  \int_0^1 B_0(s) \sum_{l=0}^d  c_{jl} dW_l(s) \right] \\
&=  \sum_{k=0}^d c_{ik}  \sum_{l=0}^d  c_{jl} \, E_{\boldsymbol{\theta}_0} \left[ \int_0^1 B_0(s)  d W_k(s)     \int_0^1 B_0(s)dW_l(s) \right] \\
&= \sum_{k=0}^d c_{ik}    c_{jk} \, E_{\boldsymbol{\theta}_0} \left[ \left( \int_0^1 B_0(s)  d W_k(s) \right)^2 \right] \\
&= \sigma_{ij} \,  \int_0^1 E_{\boldsymbol{\theta}_0} \left[ \left(  B_0(s) \right)^2  \right] ds \\
&= \sigma_{ij} \,  \int_0^1 \sigma_{00} s \, ds = \frac{\sigma_{00} \sigma_{ij}}{2}.
\end{align*}
To compute $E_{\boldsymbol{\theta}_0}[U V']$ first note that we can write it as
\begin{align*}
\left( E_{\boldsymbol{\theta}_0} [ U V'] \right)_{ij} &= E_{\boldsymbol{\theta}_0} [ U_i V_j ]  = E_{\boldsymbol{\theta}_0} \left[ B_0(1) B_i(1) \int_0^1 B_0(s) \sum_{k=0}^d c_{jk} d W_k(s) \right] \\
&= E_{\boldsymbol{\theta}_0} \left[ \sum_{m=0}^d c_{0m} W_m(1) \sum_{l=0}^d c_{il} W_l(1)  \int_0^1 \sum_{n=0}^d c_{0n} W_n(s) \sum_{k=0}^d c_{jk} d W_k(s) \right] \\
&=  \sum_{n=0}^d c_{0n} \sum_{k=0}^d c_{jk} \sum_{m=0}^d c_{0m} \sum_{l=0}^d c_{il} \, E_{\boldsymbol{\theta}_0} \left[  W_m(1)  W_l(1)  \int_0^1  W_n(s)  d W_k(s) \right] .
\end{align*}
Now, for each of the remaining expectations use the product rule $W_m(1) W_l(1) = \int_0^1 W_m(s) dW_l(s) + \int_0^1 W_l(s) dW_m(s) + 1(m=l)$ to obtain
\begin{align*}
E_{\boldsymbol{\theta}_0} \left[  W_m(1)  W_l(1)  \int_0^1  W_n(s)  d W_k(s) \right] &= E_{\boldsymbol{\theta}_0} \left[  \int_0^1 W_m(s) dW_l(s)  \int_0^1  W_n(s)  d W_k(s) \right] \\
&\hspace{5mm} + E_{\boldsymbol{\theta}_0} \left[  \int_0^1 W_l(s) dW_m(s)  \int_0^1  W_n(s)  d W_k(s) \right] \\
&= \int_0^1 E_{\boldsymbol{\theta}_0} \left[ W_m(s) W_n(s)  \right] ds \, 1(l = k) \\
&\hspace{5mm} + \int_0^1 E_{\boldsymbol{\theta}_0} \left[ W_l(s) W_n(s) \right] ds \, 1(m = k) \\
&= \int_0^1 s 1(m=n) \, ds \, 1(l = k) + \int_0^1 s 1(l = n) \, ds \, 1(m = k) \\
&= \frac{1}{2} 1( m = n) 1(l = k) + \frac{1}{2} 1(l = n) 1(m = k).
\end{align*}
Substituting in the expression for $E_{\boldsymbol{\theta}_0}[U V']$ we obtain
\begin{align*}
\left( E_{\boldsymbol{\theta}_0} [ U V'] \right)_{ij} &= \frac{1}{2} \sum_{n=0}^d c_{0n} \sum_{k=0}^d c_{jk} c_{0n}  c_{ik}  + \frac{1}{2} \sum_{n=0}^d c_{0n} \sum_{k=0}^d c_{jk}  c_{0k}  c_{in}  \\
&= \frac{1}{2} \sigma_{00} \sigma_{ij}  + \frac{1}{2} \sigma_{0i} \sigma_{0j} = \left( E_{\boldsymbol{\theta}_0} [ U V'] \right)_{ji}  .
\end{align*}

Therefore, the $(i,j)$th component, $1 \leq i,j \leq d$, of the limiting distribution's covariance matrix is given by 
\begin{align*}
\cov_{\boldsymbol{\theta}_0} (\mathcal{I})_{ij} &= \sigma_{00} \sigma_{ij} + \sigma_{0i} \sigma_{0j}  - (\sigma_{00} \sigma_{ij} + \sigma_{0i} \sigma_{0j}) + \frac{\sigma_{00} \sigma_{ij}}{2} = \frac{\sigma_{00} \sigma_{ij}}{2}. 
\end{align*}
\end{proof}

\bibliographystyle{apalike}
\bibliography{GradientBib}


\end{document}